\documentclass[11pt]{article}
\usepackage{amsmath, amssymb, amsthm,bbm}
\usepackage{indentfirst}

\usepackage{enumerate,enumitem}

\allowdisplaybreaks

\theoremstyle{plain}
\newtheorem{theorem}{Theorem}[section]
\newtheorem{lemma}[theorem]{Lemma}
\newtheorem{claim}[theorem]{Claim}
\newtheorem{proposition}[theorem]{Proposition}

\newtheorem{corollary}[theorem]{Corollary}
\newtheorem{conjecture}[theorem]{Conjecture}

\theoremstyle{definition}
\newtheorem{definition}[theorem]{Definition}
\newtheorem{remark}[theorem]{Remark}

\newcommand{\Bin}{\ensuremath{\textrm{Bin}}}

\newcommand{\eps}{\varepsilon}
\newcommand{\comp}[1]{\overline{#1}}
\newcommand{\Ex}{\mathbb{E}}

\newcommand{\DD}{\mathcal{D}}

\newcommand{\FF}{\mathcal{F}}

\newcommand{\RR}{\mathbb{R}}
\newcommand{\Success}{\mathcal{S}}
\newcommand{\TT}{\mathcal{T}}

\newcommand{\indicator}{\mathbbm{1}}
\newcommand{\taumax}{\bar{\tau}}
\newcommand{\taunorm}{\tau'}

\newcommand{\smax}{{s'}}

\newcommand{\ch}[1]{{#1\downarrow}}
\newcommand{\pre}[1]{#1^-}
\newcommand{\suc}[1]{{#1^+}}

\newcommand{\Gnp}{G_{n,p}}
\newcommand{\Gnpp}{G_{n,p'}}
\newcommand{\Gnq}{G_{n,q}}

\renewcommand{\ge}{\geqslant}
\renewcommand{\le}{\leqslant}

\title{Packing trees of unbounded degrees in random graphs}

\author{
Asaf Ferber\thanks{Department of Mathematics, Yale University, and Department of Mathematics, MIT. E-mail addresses: asaf.ferber@yale.edu and ferbera@mit.edu.}
\and
Wojciech Samotij\thanks{School of Mathematical Sciences, Tel Aviv University, Tel Aviv 6997801, Israel. E-mail address: samotij@post.tau.ac.il. Research supported by Israel Science Foundation grant 1147/14.}}

\date{\today}

\begin{document}
\maketitle

\begin{abstract}
  In this paper, we address the problem of packing large trees in $\Gnp$. In particular, we prove the following result. Suppose that $T_1, \dotsc, T_N$ are $n$-vertex trees, each of which has maximum degree at most $(np)^{1/6} / (\log n)^6$. Then with high probability, one can find edge-disjoint copies of all the $T_i$ in the random graph $\Gnp$, provided that $p \ge (\log n)^{36}/n$ and $N \le (1-\eps)np/2$ for a positive constant $\eps$. Moreover, if each $T_i$ has at most $(1-\alpha)n$ vertices, for some positive $\alpha$, then the same result holds under the much weaker assumptions that $p \ge (\log n)^2/(cn)$ and $\Delta(T_i) \le c np / \log n$ for some~$c$ that depends only on $\alpha$ and $\eps$. Our assumptions on maximum degrees of the trees are significantly weaker than those in all previously known approximate packing results.
\end{abstract}

\section{Introduction}

A collection of graphs $G_1, \dotsc, G_t$ is said to \emph{pack} into a graph $G$ if there exist edge-disjoint subgraphs $H_1, \dotsc, H_t$ of $G$ such that $H_i$ is isomorphic to $G_i$ for every $i$. The case when all the $G_i$ are trees has attracted particular interest in the last few decades. The following conjecture, known as the Tree Packing Conjecture, appears in a paper of Gy\'arf\'as and Lehel from 1976.
\begin{conjecture}[{\cite{gyarfas-lehel}}]
  \label{tpc}
  Any collection $T_1, \dotsc, T_n$ of trees with $v(T_i) = i$ for each $i$ packs into $K_n$.
\end{conjecture}
A closely related conjecture had been posed by Ringel in 1963.
\begin{conjecture}[{\cite[Problem~25]{ringel}}]
  \label{ringel}
  For every tree $T$ with $n+1$ vertices, $2n+1$ copies of $T$ pack into $K_{2n+1}$.
\end{conjecture}
If true, both conjectures would be tight. Indeed, in both cases the hypothetical embedding of the trees would have to use all the edges of the host graph. Several cases of both conjectures have been established, see, e.g., \cite{dobson1,dobson2,dobson3,gyarfas-lehel,hbk,roditty}, but they all assume special structure of the trees. Perhaps the first attempt to resolve Conjecture~\ref{tpc} for arbitrary trees is due to Bollob\'as~\cite{bollobas}, who showed that one can pack the $\lfloor n / \sqrt{2} \rfloor$ smallest trees. He also remarked that this could be improved to $\lfloor \sqrt{3}n/2 \rfloor$, provided that the notoriously difficult Erd\H{o}s--S\'os conjecture on embedding trees in graphs with large average degree is true (a solution of the conjecture was announced by Ajtai, Koml\'os, Simonovits, and Szemer\'edi in the early 1990s). At the other end of the spectrum, Balogh and Palmer~\cite{balogh-palmer} have recently proved that one can pack the $\lfloor n^{1/4}/10 \rfloor$ largest trees into $K_{n+1}$, that is, if one is allowed to use an extra vertex. Moreover, they have shown that if one bounds the maximum degree of the trees by $2n^{2/3}$, then one can pack the $\lfloor n^{1/3}/4 \rfloor$ largest trees into $K_n$. It thus appears that packing large trees is a much harder task than packing small ones. On the other hand, it seems that imposing bounds on the maximum degrees of the trees makes the problem more tractable.

Following this direction, B\"ottcher, Hladk\'y, Piguet, and Taraz~\cite{bhpt} showed that if $T_1, \dotsc, T_N$ are trees with at most $n$ vertices and maximum degree bounded by a constant, then they pack into $K_{\lceil(1+\eps)n\rceil}$, provided that $e(T_1) + \ldots + e(T_N) \le (1-\eps)\binom{n}{2}$. Generalising this result, Messuti, R\"odl, and Schacht~\cite{mrs} proved that the same conclusion holds under the weaker assumption that all the $T_i$ belong to some fixed minor-closed family. Recently, Ferber, Lee, and Mousset~\cite{FLM} improved this result by showing that these graphs can be packed into $K_n$. Even more recently, Kim, K\"uhn, Osthus, and Tyomkyn~\cite{KKOT} extended the result of~\cite{FLM} to arbitrary graphs with bounded maximum degree. These developments imply the following approximate versions of Conjectures~\ref{tpc} and \ref{ringel}.

\begin{corollary}
  For all positive $\eps$ and $\Delta$, if $n$ is sufficiently large, then:
  \begin{enumerate}[label={\rm(\textit{\roman*})}]
  \item
    Every collection $T_{\lceil \eps n \rceil},\dotsc, T_n$ of trees with $v(T_i) = i$ and $\Delta(T_i) \le \Delta$ packs into $K_n$.
  \item
    At least $(2-\varepsilon)n$ copies of each tree $T$ with $v(T) = n+1$ and $\Delta(T) \le \Delta$ pack into $K_{2n+1}$.
  \end{enumerate}
\end{corollary}

While we were writing these lines, Joos, Kim, K\"uhn, and Osthus \cite{joos} announced a proof of the Tree Packing Conjecture for all bounded degree trees.

In this paper, we strengthen the result of B\"ottcher et al.\ in a somewhat different direction. We address the problem of packing trees with unbounded maximum degree into a random host graph. We work with the usual binomial random graph $\Gnp$, that is, the graph obtained from the complete graph $K_n$ by keeping each edge with probability $p$, independently at random.

Our first result addresses the problem of packing a collection of spanning trees. We show that a.a.s.\ (\emph{asymptotically almost surely}) one can pack into $\Gnp$ a given collection of $n$-vertex trees, provided that the total number of edges of these trees does not exceed $(1-\eps)$-proportion of the (expected) number of edges of the host graph and the maximum degree of each tree in the collection is bounded by a small power of the (expected) average degree of the host graph.

\begin{theorem}
  \label{theorem:spanning}
  Let $\eps$ be a positive constant and suppose that $p \ge (\log n)^{36} / n$ and $N \le (1-\eps)np/2$. If $T_1, \dotsc, T_N$ are $n$-vertex trees with maximum degree at most $(np)^{1/6} / (\log n)^6$, then a.a.s.\ $T_1,\dotsc,T_N$ pack into $\Gnp$.
\end{theorem}

Note that unlike in some of the previously mentioned results, in Theorem~\ref{theorem:spanning} we assume that all the trees have the same size. Even though it might seem somewhat restrictive, one may always ``cut and paste'' the trees together. For example, in the setting of Conjecture~\ref{tpc}, one may merge $T_{i+1}$ with $T_{n-i}$ by identifying two arbitrarily chosen leaves to obtain a tree with $n$ vertices, whose maximum degree does not exceed $\max\{\Delta(T_{i+1}), \Delta(T_{n-i})\}$. Therefore, it seems natural to determine conditions guaranteeing that a large collection of $n$-vertex trees packs into $K_n$, or perhaps into $K_{(1+\eps)n}$. In our setting, as the host graph is $\Gnp$, which typically has about $\binom{n}{2}p$ edges, one cannot expect to pack more than $np/2$ spanning trees. Moreover, as $\Gnp$ is a.a.s.\ not connected unless $p \ge (\log n + \omega(1))/n$, one can see that our Theorem~\ref{theorem:spanning} is approximately optimal with respect to both the number of trees~$N$ (up to a $1-o(1)$ multiplicative factor) and the edge probability~$p$ (up to a polylogarithmic factor).

Our second result addresses the problem of packing a collection of almost spanning trees. In this case, we can pack trees with much larger maximum degrees.

\begin{theorem}
  \label{theorem:almost-spanning}
  Let $\eps$ be a positive constant and suppose that $p \gg (\log n)^2/n$ and $N \le (1-\eps)np/2$. If $T_1, \dotsc, T_N$ are trees, each of which has at most $(1-\eps)n$ vertices and maximum degree at most $(\eps/8)^8 np/\log n$, then a.a.s.\ $T_1, \dotsc, T_N$ pack into $\Gnp$.
\end{theorem}

At the heart of the proof of both Theorems~\ref{theorem:spanning} and~\ref{theorem:almost-spanning} lies the following technical generalisation of the latter, which is the main result of this paper. While the derivation of Theorem~\ref{theorem:almost-spanning} from this result is fairly straightforward, our proof of Theorem~\ref{theorem:spanning} require several additional ingredients. Therefore, we postpone both arguments to Section~\ref{sec:derivation}.

\begin{theorem}
  \label{theorem:main}
  Suppose that $\alpha, \eps, p \in (0,1/2)$ and integers $\Delta$ and $n$ satisfy
  \begin{equation}
    \label{eq:main-assumptions}
    \frac{150 (\log n)^2}{\alpha \eps n} \le p \le \frac{\eps \alpha^4}{128} \quad \text{and} \quad \Delta \le \min\left\{ \alpha, \frac{\eps}{\log (1/\alpha)} \right\} \cdot \frac{\eps np}{1600\log n}.
  \end{equation}
  Suppose that $T_1,\dotsc, T_N$, where $N \le (1-\eps)np/2$, is a collection of trees, each of which has at most $(1-\alpha)n$ vertices and maximum degree at most $\Delta$. Then with probability at least $1-n^{-7}$, the trees $T_1, \dotsc, T_N$ pack into $\Gnp$.
\end{theorem}

\begin{remark}
  \label{remark:Ui-distribution}
  In the proof of Theorem~\ref{theorem:spanning}, we shall need the following additional property of the packing whose existence is guaranteed by Theorem~\ref{theorem:main}, which we establish in the course of its proof. Denote the edge-disjoint embeddings of $T_1, \dotsc, T_N$ by $\varphi_1, \dotsc, \varphi_N$, respectively. For each $s \in [N]$, let $v_s$ be an arbitrary vertex of $T_s$ and let $W_s$ be the set of vertices of $\Gnp$ not covered by $\varphi_s(T_s)$ plus the vertex $\varphi_s(v_s)$. If we additionally assume that $p \ge 30\log n / (\alpha^2 n)$, then with probability at least $1 - n^{-7}$, for every pair of distinct vertices $u$ and $w$,
  \[
  \sum_{s = 1}^N \indicator[\{u, w\} \subseteq W_s] \cdot \frac{1}{|W_s|} \le \frac{2p}{n} \cdot \max_s |W_s|.
  \]
\end{remark}

The proof of Theorem~\ref{theorem:main} utilises and extends the ``online sprinkling" technique introduced by the first author and Vu~\cite{FV}. Roughly speaking, we embed our trees and expose $\Gnp$ together, edge by edge, making sure that the trees are embedded disjointly and each discovered edge of $\Gnp$ is used in the embedding.

The proof of Theorem~\ref{theorem:spanning} relies heavily on ideas from and elegant short paper of Krivelevich~\cite{Kriv} and a beautiful recent paper of Montgomery~\cite{montgomery2014embedding}, which essentially resolves the problem of embedding a given $n$-vertex tree into $\Gnp$. We shall actually require a slight generalisation of the main results of~\cite{Kriv,montgomery2014embedding}, Theorem~\ref{theorem:Montgomery} below. It determines a sufficient condition for the edge probability that ensures that with high probability a given $n$-vertex tree can be embedded into a random subgraph of an $n$-vertex graph that is almost complete. Even though one may derive Theorem~\ref{theorem:Montgomery} by carefully following the arguments of~\cite{Kriv,montgomery2014embedding}, with some obvious modifications, we shall do it in full detail in Section~\ref{sec:embedding-spanning-trees} for the convenience of the reader. Following standard practice, given a graph $G$ and a $p \in [0,1]$, we shall denote by $G_p$ the random subgraph of $G$ obtained by keeping each edge with probability $p$, independently of other edges, which we shall refer to as the \emph{$p$-random subgraph of $G$}.

\begin{theorem}
  \label{theorem:Montgomery}
  Let $T$ be a tree with $n$ vertices and maximum degree $\Delta$ and let $v$ be an arbitrary vertex of $T$. Let $G$ be an $n$-vertex graph with $\delta(G)\ge n-n/(\Delta(\log n)^5)$ and let $x \in V(G)$. If $p \ge \Delta(\log n)^5/n$ and $n \ge n_0$ for some absolute constant $n_0$, then with probability at least $1-n^{-3}$, there is an embedding of $T$ into $G_p$ that maps $v$ to $x$.
\end{theorem}

\subsection{Outline of our paper}

Our paper is organised as follows. In Section~\ref{sec:pre}, we describe several auxiliary results which we shall later use in our arguments. In particular, in Section~\ref{sec:conc}, we present the main concentration inequality that is used in our proofs. In Section~\ref{sec:par}, we prove a useful lemma about partitioning a tree into two subtrees. In Sections~\ref{sec:bare}, we relate the number of leaves in a tree to the number of its so-called long bare paths. In Section~\ref{sec:BFS}, we describe an ordering of the vertices of a tree related to the Breadth First Search algorithm. In Section~\ref{sec:expans-rand-graphs}, we establish basic expansion properties of a typical $\Gnp$. In Section~\ref{sec:embedding-spanning-trees}, we show how to modify the main results of~\cite{Kriv,montgomery2014embedding} in order to obtain Theorem~\ref{theorem:Montgomery}. In Section~\ref{sec:derivation}, we derive Theorems~\ref{theorem:spanning} and~\ref{theorem:almost-spanning} from Theorems~\ref{theorem:main} and Theorem~\ref{theorem:Montgomery}. Finally, in Section~\ref{sec:proof of main}, we prove our main result, Theorem~\ref{theorem:main}. To this end, we describe a randomised algorithm that tries to embed a given collection of trees randomly while generating $\Gnp$ at the same time. Our main goal is to show that each tree in the collection is embedded in a somewhat uniform fashion. The precise statement is given by Lemma~\ref{lemma:uniform-dist}, which is the heart of our argument (and the most technical part of this paper). We close the paper with several concluding comments and remarks, in Section~\ref{sec:concluding-remarks}.

\section{Preliminaries} \label{sec:pre}

\subsection{A concentration result} \label{sec:conc}

In our proofs, we shall make use of the following straightforward generalisation of Bennett's inequality~\cite{Be62} (see also~\cite[Chapter~2]{BoLuMa13}) to sums of weakly dependent random variables. Since this generalisation can be proved using the ideas of~\cite{Be62} and a standard Azuma-type estimate, we postpone the proof to Appendix~\ref{sec:proof-Bennett-plus}.

\begin{lemma}
  \label{lemma:Bennett-plus}
  Let $X_1, \dotsc, X_N$ be real-valued random variables such that
  \[
  0 \le X_i \le M, \qquad \Ex[X_i \mid X_1, \dotsc, X_{i-1}] \le \mu, \qquad \text{and} \qquad \Ex[X_i^2 \mid X_1, \dotsc, X_{i-1}] \le \sigma^2
  \]
  for every $i \in [N]$ and some $M$, $\mu$, and $\sigma$. Then for every positive $t$,
  \[
  \Pr\left( \sum_{i=1}^N X_i \ge N\mu + t \right) \le \exp\left(- \frac{t^2}{2(N\sigma^2 + Mt/3)}\right).
  \]
\end{lemma}

\subsection{Partitioning trees} \label{sec:par}

In the proof of Theorem~\ref{theorem:spanning}, we shall require the following folklore result about partitioning trees with bounded degree into subtrees.

\begin{lemma}
  \label{lemma:preparing-the-trees}
  Let $T$ be a tree with $n$ vertices and maximum degree $\Delta$. For every $0 \le \alpha < 1$, there are subtrees $S$ and $L$ of $T$ sharing exactly one vertex such that $E(T) = E(L) \cup E(S)$ and $\lfloor \alpha n\rfloor +1\le |V(S)| \le 2\lfloor \alpha n\rfloor$.
\end{lemma}

\begin{proof}
  Root $T$ at an arbitrary vertex $r$. For every vertex $u$ of $T$, denote by $T(u)$ the subtree of $T$ rooted at $u$ and let $|T(u)|$ denote the number of vertices of $T(u)$. (In other words, $T(u)$ is the subtree of $T$ induced by all vertices $w$ for which the unique path from $r$ to $w$ contains $u$, including $u$ itself.) Let $v$ be a vertex of maximum distance from $r$ among all vertices satisfying $|T(v)|>\lfloor\alpha n\rfloor$. Note that such a vertex exists as by the assumption $\alpha<1$ we have that the root $r$ satisfies $|T(r)|>\lfloor\alpha n\rfloor$. Let $u_1,\dotsc,u_d$ be the children of $v$ in $T$. Observe that for every $j \in \{1,\dotsc,d\}$ we have $|T(u_j)|\le \lfloor \alpha n\rfloor $ and that, since $|T(v)|>\lfloor \alpha n\rfloor$, also $\sum_j |T(u_j)| \ge  \lfloor \alpha n\rfloor$. Let $1\le i \le d$ be the smallest index for which $s := \sum_{j=1}^i|T(u_j)|\ge \lfloor \alpha n\rfloor$. Clearly, $s \le 2\lfloor \alpha n\rfloor -1$, and therefore, we may let $S$ be the subtree of $T$ induced by the set $\{v\} \cup \bigcup_{j=1}^i V(T(u_j))$ and $L = T - T(u_1) - \ldots - T(u_i) = T - (V(S) \setminus \{v\})$.
\end{proof}

\subsection{Bare paths versus leaves} \label{sec:bare}

In the proof of Theorem~\ref{theorem:Montgomery}, we shall use the following lemma due to Krivelevich~\cite{Kriv}, which relates the number of leaves in a tree to the number of its long bare paths. A \emph{bare path} in a tree $T$ is a path whose all inner vertices are of degree exactly two in $T$.

\begin{lemma}[{\cite{Kriv}}]
  Let $k$, $\ell$, and $n$ be postive integers and let $T$ be a tree with $n$ vertices. If $T$ has at most $\ell$ leaves, then it contains a collection of at least $n/(k+1)-(2\ell-2)$ vertex-disjoint bare paths of length~$k$ each.
\end{lemma}

We shall only invoke the above lemma in the following form, setting $\ell=n/4k$.

\begin{corollary}
  \label{cor:many-bare}
  Let $n$ and $k$ be positive integers. A tree with $n$ vertices has either at least $n/4k$ leaves or a~collection of at least $n/4k$ vertex-disjoint bare paths of length $k$ each.
\end{corollary}

\subsection{Breadth First Search ordering}

\label{sec:BFS}

In two of our proofs, we shall be considering the ordering of the vertices of an $m$-vertex tree $T$, rooted at an arbitrary vertex $v_0$, as $v_0, \dotsc, v_{m-1}$, according to the time of the first visit of the \emph{Breadth First Search} algorithm (\emph{BFS} for short) executed on $T$, rooted at $v_0$. For more details on the BFS algorithm, we refer the reader to~\cite[Page 99]{We96}. In the sequel, we shall call this ordering the \emph{BFS ordering} of $T$. We now note the following two simple properties of this ordering:
\begin{enumerate}[label={\rm(\textit{O\arabic*})}]
\item
  \label{item:T-ordering-1}
  The children of each vertex $v_i$ are assigned consecutive labels larger than $i$.
\item
  \label{item:T-ordering-2}
  If $i_1 < i_2$, then the children of $v_{i_1}$ appear before the children of $v_{i_2}$ in the ordering.
\end{enumerate}
Finally, let $J \subseteq \{0, \dotsc, m-1\}$ be the set of indices of all non-leaf vertices (including the root $v_0$, even if $v_0$ has degree one in $T$). Moreover, for each $i \in J$:
\begin{itemize}
\item
  Let $d_i$ be the number of children of $v_i$; that is, $d_0 = \deg_Tv_0$ and $d_i = \deg_Tv_i - 1$ for $i > 0$.
\item
  Let $\ch{i}$ be the smallest label of a child of $v_i$; the children of $v_i$ are $v_{\ch{i}}, \dotsc, v_{\ch{i} + d_i-1}$.
\item
  Let $\pre{i}$ and $\suc{i}$ be the largest label in $J$ that is smaller than $i$ (the predecessor of $i$ in $J$) and the smallest label in $J$ that is larger than $i$ (the successor of $i$ in $J$), respectively.
\end{itemize}
Observe that for each $i \in J \setminus \{0\}$,
\[
\bigcup_{j < i, j \in J} \big( \{v_j\} \cup N_T(v_j) \big) = \{v_0, \ldots, v_{\ch{i}-1}\}.
\]

\subsection{Expansion in random graphs}
\label{sec:expans-rand-graphs}

In the proof of Theorem~\ref{theorem:Montgomery}, we shall rely on some basic facts about expansion properties of random graphs, stated in Proposition~\ref{prop:expansion} below. Since we shall be working with random subgraphs of almost complete graphs rather than the usual $\Gnp$ model, we include a (standard) proof of these facts.

\begin{definition}
  Given a graph $G$ and a set $W \subseteq V(G)$, we say that $G$ \emph{$d$-expands} into $W$ if
  \begin{enumerate}[label={\rm(\textit{E\arabic*})}]
  \item
    \label{item:expansion-small}
    $|N_G(X,W)|\ge d|X|$ for all $X\subseteq V(G)$ with $1\le |X| < \frac{|W|}{2d}$ and
  \item
    \label{item:expansion-large}
    $e_G(X,Y) > 0$ for all disjoint $X,Y\subseteq V(G)$ with $|X|, |Y| \ge \frac{|W|}{2d}$.
  \end{enumerate}
\end{definition}

\begin{proposition}
  \label{prop:expansion}
  Let $n$ and $w$ be positive integers, let $G$ be an $n$-vertex graph, and let $W$ be a set of $w$ vertices of $G$. Let $d \ge 1$ and suppose that $\delta(G) \ge n - w/(8d)$. If $p \ge 500d\log n/w$, then with probability at least $1 - n^{-7}$, the random graph $G_p$ $d$-expands into $W$.
\end{proposition}
\begin{proof}
  Let $G$, $W$, $d$, $p$, and $w$ be as in the statement of the proposition. Note first that every vertex of $G$ has at least $7w/8$ neighbours in the set $W$. Therefore, standard estimates on tail probabilities of binomial random variables (such as Lemma~\ref{lemma:Bennett-plus}) imply that for every $v \in V(G)$,
  \[
  \Pr\left(\deg_{G_p}(v, W) < 2wp/3 \right) \le \exp\left(-wp/50\right) < n^{-9}.
  \]
  In particular, with probability at least $1-n^{-8}$, each vertex of $G_p$ has at least $2wp/3$ neighbours in~$W$. Assuming that this event holds, if~\ref{item:expansion-small} fails, then there are sets $X, Y \subseteq V(G)$ such that
  \[
  \frac{2wp}{3d} \le |X| < \frac{w}{2d}, \qquad |Y| < d|X|, \quad \text{and} \quad e_{G_p}(X,Y) \ge \frac{2|X|wp}{3} \ge \Ex[e_{G_p}(X,Y)] + \frac{|X|wp}{6}.
  \]
  By standard estimates on tail probabilities of binomial random variables (such as Lemma~\ref{lemma:Bennett-plus}) and the union bound, the probability $P$ of this event satisfies
  \[
  P \le \sum_{x \ge 2wp/(2d)} \binom{n}{x} \binom{n}{dx} \exp\left(- \frac{xwp}{40}\right) \le \sum_{x \ge 2wp/(2d)} \exp\left(dx \cdot \left(2\log n - \frac{wp}{40d}\right)\right) < n^{-8}.
  \]
  Finally, the probability $Q$ that~\ref{item:expansion-large} fails may be bounded from above as follows:
  \[
  Q \le \binom{n}{w/(2d)}^2 (1-p)^{w^2/(4d^2)} \le \exp\left( \frac{w}{d} \cdot \left(\log n - \frac{wp}{4d}\right)\right) < n^{-8}.
  \]
  This completes the proof of the proposition.
\end{proof}

\section{Embedding spanning trees into random graphs}

\label{sec:embedding-spanning-trees}

In this section, we consider the problem of embedding an $n$-vertex tree $T$ into a random subgraph of an $n$-vertex graph that is nearly complete and prove Theorem~\ref{theorem:Montgomery}. As in many previous works on embedding trees in random graphs, we shall distinguish two cases, depending on the number of leaves of $T$. First, we deal with the easier case when $T$ contain at least $n/(\log n)^3$ many leaves, which was resolved several years ago by Krivelevich~\cite{Kriv}. Our argument here closely follows that of~\cite{Kriv}, with a few minor modifications.

\begin{theorem}
  \label{theorem:many-leaves}
  Let $T$ be a tree with $n$ vertices and maximum degree $\Delta$. Suppose that $T$ has at least $n/(\log n)^3$ many leaves and let $v$ be an arbitrary vertex of $T$. Let $G$ be an $n$-vertex graph with $\delta(G)\ge n-n/(\Delta(\log n)^5)$ and let $x \in V(G)$. If $p\ge \Delta(\log n)^5/n$ and $n \ge n_0$ for some absolute constant $n_0$, then with probability at least $1-n^{-4}$, there is an embedding of $T$ into $G_p$ that maps $v$ to $x$.
\end{theorem}

\begin{proof}
  Given a tree $T$ and a $v \in V(T)$ as in the statement of the theorem, let $L$ denote a set of exactly $n/(2 (\log n)^3)$ leaves of $T$ such that $v \not\in L$ and let $M$ denote the set of parents of the leaves in $L$. Let $T' = T - L$ and let $m = n-|L|$. Let $v_0, \dotsc, v_{m-1}$ be the BFS ordering of $T'$ with $v_0 = v$ and let $J$ and $(d_i)_{i \in J}$ be as in Section~\ref{sec:BFS}. Suppose that $p \ge \Delta(\log n)^5/n$ and let $G$ be an $n$-vertex graph with minimum degree at least $n - n/(\Delta(\log n)^5)$. We shall show that with probability at least $1 - n^{-4}$, there is an embedding $\varphi$ of $T$ into $G_p$ satisfying $\varphi(v) = x$, provided that $n \ge n_0$ for some absolute constant $n_0$, which we shall from now on tacitly assume. Let $q$ be the unique positive real defined by $1-p = (1-q)^2$ and note that $q \ge p/2$. As $G_p$ has the same distribution as the union of two independent copies of $G_q$, we may construct the embedding in two stages. First, we show that with probability at least $1 - n^{-5}$, there is an embedding $\varphi$ of $T'$ into the first copy of $G_q$ satisfying $\varphi(v) = x$. Second, we show that with probability at least $1 - n^{-5}$, we can embed all the leaves in $L$ using the edges between the sets $\varphi(M)$ and $V(G) \setminus \varphi(V(T'))$ in the second copy of $G_q$. This is equivalent to finding an appropriate generalised matching in the $q$-random subgraph of the bipartite subgraph of $G$ induced by some two sets of sizes $|M|$ and $|L|$, respectively.

\smallskip
\noindent
\textbf{Stage 1.}
This stage consists of $|J|$ rounds, indexed by the elements of $J$; in round $i \in J$, we wish to embed the children of $v_i$. We start with $\varphi$ being the empty map and set $\varphi(v_0) = x$. Suppose that we are at the beginning of round $i$ and $v_0, \dotsc, v_{\ch{i}-1}$ are already embedded. We wish to embed $v_{\ch{i}}, \dotsc, v_{\ch{i}+d_i-1}$, the children of $v_i$. To this end, let $U_i = V(G) \setminus \varphi(\{v_0,\ldots,v_{\ch{i}-1}\})$ and expose all edges of $G_q$ between $\varphi(v_i)$ and $U_i$. (Note that each of these edges is being exposed for the first time.)  Denote their number by $X_i$. If $X_i \ge d_i$, we may map $v_{\ch{i}}, \dotsc, v_{\ch{i}+d_i-1}$ to arbitrarily chosen $d_i$ neighbours of $\varphi(v_i)$ in $U_i$ and proceed to the next round. Since $X_i \sim \Bin(|U_i|, q)$ and
\[
|U_i| \ge \delta(G) - i \ge \delta(G) - m = |L| - (n-\delta(G)) \ge n/(2(\log n)^3) - n/(\Delta(\log n)^5),
\]
standard estimates on tail probabilities of binomial random variables (such as Lemma~\ref{lemma:Bennett-plus}) yield
\[
\Pr(X_i < d_i) \le \Pr( X_i < \Delta ) \le \Pr(X_i < q|U_i| / 2) = \exp(-q|U_i|/10) \le n^{-6}.
\]
In particular, the probability that we fail to embed $T'$ into $G_q$ is at most $|J|n^{-6}$.

\smallskip
\noindent
\textbf{Stage 2.}
Let $M' =\varphi(M)$ and $L' =V(G) \setminus \varphi(V(T'))$. Our goal in this stage is to complete the embedding by finding images for the leaves in $L$ in the set $L'$. Let $B$ denote the bipartite subgraph of $G$ induced by the sets $M'$ and $L'$. The embedding $\varphi$ can be completed if and only if the graph $B_q$ contains a generalised matching, where each vertex $y \in M'$ has $d_y := \deg_T (\varphi^{-1}(y), L)$ neighbours in $L'$. Construct an auxiliary graph $B'$ by blowing up each vertex $y \in M'$ into a set $A_y$ of $d_y$ vertices, replacing each edge $yz$ of $B$ with the complete bipartite graph between $A_y$ and $z$. Let $r = q/\Delta$ and note that $1-q \le (1-r)^{d_y}$ for each $y \in M'$, as $d_y \le \Delta$ by our assumption on $T$. In particular, if we let $B^*$ be the random subgraph of $B$ such that $yz \in B^*$ if and only if $y'z \in B_r'$ for some $y' \in A_y$, then there is an obvious coupling of $B^*$ and $B_q$ such that $B^* \subseteq B_q$. It follows that $B_q$ contains the required generalised matching if and only if the graph $B_r'$ contains a perfect matching. By construction,
\[
\delta(B') \ge |L| - \Delta(n - \delta(G)) \ge \left(1 - \frac{2}{(\log n)^2} \right) |L|.
\]
As $r = p / (2\Delta) \ge  (\log n)^2 / (4|L|)$, a standard argument (see, e.g., \cite[Claim~3.6]{Robust} or \cite[Theorem~2.3]{sudakov-vu-resilience}) combined with Proposition~\ref{prop:expansion} (applied twice, once with $W \leftarrow L$ and once with $W \leftarrow V(B') \setminus L$, to the graph obtained from $B'$ by adding to it all $2\binom{|L|}{2}$ edges contained in either $L$ or $V(B') \setminus L$) shows that with probability at least $1 - n^{-5}$, the graph $B_r'$ satisfies Hall's condition.
\end{proof}

Second, we deal with trees $T$ which contain fewer than $n/(\log n)^3$ leaves. Our argument here closely follows that of Montgomery~\cite{montgomery2014embedding}, with a few minor modifications.

\begin{theorem}
  \label{theorem:many-paths}
  Let $T$ be a tree with $n$ vertices and maximum degree $\Delta$. Suppose that $T$ has at most $n/(\log n)^3$ many leaves and let $v$ be an arbitrary vertex of $T$. Let $G$ be an $n$-vertex graph with $\delta(G)\ge n-n/(\Delta(\log n)^5)$ and let $x \in V(G)$. If $p \ge \Delta(\log n)^5/n$ and $n \ge n_0$ for some absolute constant $n_0$, then with probability at least $1-n^{-3}$, there is an embedding of $T$ into $G_p$ that maps $v$ to $x$.
\end{theorem}

The main ingredient in the proof of Theorem~\ref{theorem:many-paths} is the following theorem due to Montgomery~\cite{montgomery2014embedding}, which enables one to find vertex-disjoint paths connecting given pairs of vertices in a graph with good expansion properties.

\begin{theorem}[{\cite[Theorem~4.3]{montgomery2014embedding}}]
  \label{theorem:connecting-lemma}
  Let $n$ be a sufficiently large integer and suppose that $\ell$ is a divisor of $n$ satisfying $\ell\ge 10^3 (\log n)^2$. Let $G$ be an $n$-vertex graph, let $\{(x_i,y_i) \colon 1 \le i \le n/\ell\}$ be a collection of pairwise disjoint vertex pairs, and let $W = V(G) \setminus \bigcup_i\{x_i,y_i\}$. Let $d = 10^{10} (\log n)^4 / (\log \log n)$ and suppose that $G$ $d$-expands into $W$. Then one can cover the vertex set of $G$ with $n/\ell$ vertex-disjoint paths $P_1, \ldots, P_{n/\ell}$ of length $\ell-1$ each, so that each $P_i$ has endpoints $x_i$ and $y_i$.
\end{theorem}

\begin{proof}[{Proof of Theorem~\ref{theorem:many-paths}}]
  Let $T$ and $v$ be as in the statement of the theorem and assume that $n \ge n_0$ for some absolute constant $n_0$. Since $T$ has at most $n/(\log n)^3$ many leaves, Corollary~\ref{cor:many-bare} implies that it must contain at least $5n/(4(\log n)^3)$ vertex-disjoint bare paths of length $\ell := (\log n)^3/5$ each. In particular, there is a collection $\{P_i \colon 1 \le i \le n/(\log n)^3\}$ of such paths, none of which contains $v$. Replace each such path with an edge (by removing all the interior vertices) to obtain a tree $T'$ with at most $5n/6$ vertices.

  Let $q$ be the unique positive real defined by $1-p = (1-q)^2$ and note that $q \ge p/2$. As $G_p$ has the same distribution as the union of two independent copies of $G_q$, we may construct an embedding of $T$ into $G_p$ in two stages. First, we show that with probability at least $1 - n^{-4}$, there is an embedding $\varphi$ of $T'$ into the first copy of $G_q$ satisfying $\varphi(v) = x$. Second, we show that with probability at least $1 - n^{-4}$, using the edges of the second copy of $G_q$, we may connect the endpoints of all the $P_i$ by vertex-disjoint paths (of length $\ell$ each) covering the set $V(G) \setminus \varphi(V(T'))$, which completes the embedding.

\smallskip
\noindent
\textbf{Stage 1.}
We proceed exactly as in Stage 1 of the proof of Theorem~\ref{theorem:many-leaves}, obtaining the required embedding $\varphi$ of $T'$ into the first copy of $G_q$ with probability at least $1 - n^{-5}$.

\smallskip
\noindent
\textbf{Stage 2.}
Let $W = V(H) \setminus \varphi(V(T'))$ and let $\{(x_i,y_i)\}_i$ be the collection of endpoints of all the~$P_i$. Let $U = W \cup \bigcup_i \{x_i,y_i\}$ and let $G' = G[U]$. Since $|W| \ge n/6$ and
\[
\delta(G') = |U| - (n - \delta(G)) \ge |U| - n/(\Delta(\log n)^5),
\]
Proposition~\ref{prop:expansion} invoked with $d = 10^{10} (\log n)^4 / (\log \log n)$ implies that with probability at least $1 - n^{-6}$, the graph $G_q'$ $d$-expands into the set $W$. It now follows from Theorem~\ref{theorem:connecting-lemma} that with such high probability, we may complete the embedding $\varphi$ using the edges of the second copy of $G_q$.
\end{proof}

\section{Derivation of Theorems~\ref{theorem:spanning} and~\ref{theorem:almost-spanning}}

\label{sec:derivation}

\begin{proof}[Derivation of Theorem \ref{theorem:almost-spanning}]
  First, we may assume that $n$ is sufficiently large and that $\eps < 1$ or otherwise there is nothing to prove. Suppose that $p \gg (\log n)^2/n$ and $N \le (1-\eps)np/2$ and let $T_1, \dotsc, T_N$ be trees satisfying the assumptions of the theorem. If $p \le \eps^5 / 2^{12}$, then the assertion of the theorem follows directly from Theorem~\ref{theorem:main} invoked with $\eps \leftarrow \eps/2$ and $\alpha \leftarrow \eps/2$. Therefore, we shall assume that $p > \eps^5 / 2^{12}$. Let $K$ be the smallest integer for which $p / K \le \eps^5 / 2^{12}$, and observe that $K \le 2^{12} / (\eps^5)$. Let $c \colon E(K_n)\rightarrow [K]$ be a random colouring of the edges of $K_n$ obtained by assigning to each edge a uniformly chosen element of $[K]$, independently of other edges.

  Now, for each $i \in [K]$, let $G_i$ denote the random subgraph of $\Gnp$ comprising all edges that the random map $c$ assigned the color $i$. Observe that each $G_i$ is distributed as $\Gnq$, where $q = p/K$. Let us partition the collection $T_1,\dotsc,T_N$ into $K$ disjoint batches, denoted $\TT_1, \dotsc, \TT_K$, each of which contains $\lfloor N/K \rfloor$ or $\lceil N/K \rceil$ trees.

  Finally, invoke Theorem~\ref{theorem:main} with $p \leftarrow q$, $\eps \leftarrow \eps/2$, $\alpha \leftarrow \eps/2$, and the collection $T_1, \dotsc, T_N$ replaced by $\TT_i$ for each $i \in [K]$ to conclude that with probability at least $1 - n^{-7}$, all trees in $\TT_i$ pack into $G_i$. By the union bound,
  \[
  \Pr\big(\text{$\TT_i$ cannot be packed into $G_i$ for some $i$}\big) \le Kn^{-7} \le n^{-6}.
  \]
  As $G_1, \dotsc, G_K$ are edge-disjoint subgraphs of $\Gnp$, this completes the proof.
\end{proof}

An argument analogous to the one given above can be used to derive Theorem~\ref{theorem:spanning} from the following, seemingly weaker, statement.

\begin{theorem}
  \label{theorem:spanning-weaker}
  Let $\eps$ be a positive constant and suppose that $(\log n)^{12}/n \le p \le n^{-2/3}$ and $N \le (1-\eps)np/2$. If $T_1, \dotsc, T_N$ are $n$-vertex trees with maximum degree at most $(np)^{1/2}/(\log n)^6$, then with probability at least $1-2n^{-2}$, the trees $T_1,\dotsc,T_N$ pack into $\Gnp$, provided that $n \ge n_0$ for some absolute constant $n_0$.
\end{theorem}

\begin{proof}
  Fix a positive $\eps$, suppose that $p$ and $N$ satisfy the assumptions of the theorem, let $\Delta = (np)^{1/2}/(\log n)^6$, and let $T_1, \dotsc, T_N$ be $n$-vertex trees with maximum degree at most $\Delta$. Furthermore, suppose that $n \ge n_0$ for some sufficiently large absolute constant $n_0$. Let $\alpha = \eps / (8 \Delta (\log n)^5)$ and for each $s \in [N]$, invoke Lemma~\ref{lemma:preparing-the-trees} to find a partition of the edges of $T_s$ into two subtrees $L_s$ and $S_s$ that share precisely one vertex, denoted $v_s$, and satisfy $|V(L_s)| \le (1-\alpha) n$ and $\alpha n \le |V(S_s)| \le 2 \alpha n$.

  Let $q = \eps p /2$, let $p'$ be the unique positive real satisfying $1 - p = (1-q)(1-p')$, and note that $p' \ge (1-\eps/2)p$. As $\Gnp$ has the same distribution as the union of independent copies of $\Gnpp$ and $\Gnq$, we may construct the edge-disjoint embeddings $\varphi_1, \dotsc, \varphi_N$ of $T_1, \dotsc, T_N$ into $\Gnp$ in two stages. First, using Theorem~\ref{theorem:main}, we show that with probability at least $1 - n^{-6}$, the trees $L_1, \dotsc, L_N$ pack into $\Gnpp$ in a certain uniform fashion which we specify below. Second, using Theorem~\ref{theorem:Montgomery}, we show that with probability at least $1-n^{-2}$, the edges of $\Gnq$ that were not covered by the packing of the $L_s$ can be used to extend this packing to a packing of the $T_s$ by appropriately embedding the $S_s$.

  \smallskip
  \noindent
  \textbf{Stage 1.}
  Since each $L_s$ is a tree with at most $(1-\alpha)n$ vertices and maximum degree at most $\Delta$ and $N \le (1-\eps)np/2 \le (1-\eps/2)np'/2$, we may invoke Theorem~\ref{theorem:main} with $\eps \leftarrow \eps/2$ to conclude that with probability at least $1 - n^{-7}$, there exist pairwise edge-disjoint embeddings $\varphi_1, \dotsc, \varphi_N$ of the trees $L_1, \dotsc, L_N$, respectively, into the graph $\Gnpp$. Denote by $W_s$ the set of vertices of $\Gnpp$ not covered by $\varphi_s(L_s)$ plus the vertex $\varphi_s(v_s)$ and observe that
  \begin{equation}
    \label{eq:Ui-size}
    \alpha n \le |W_s| = |V(S_s)| \le 2 \alpha n.
  \end{equation}
  As $p \ge 30\log n/(\alpha^2n)$, we may additionally assume that the sets $W_s$ are somewhat uniformly distributed, that is,
  \begin{equation}
    \label{eq:Ui-distribution}
    \sum_{i = 1}^N \indicator[\{x,y\} \subseteq W_s] \cdot \frac{1}{|W_s|} \le \frac{2p}{n} \cdot \max_s |W_s| \le 4 \alpha p,
  \end{equation}
  see Remark~\ref{remark:Ui-distribution}. Last but not least, let $H_1$ denote the union of all $\varphi_s(L_s)$. Since clearly $H_1 \subseteq \Gnpp$ and $p'n \gg \log n$, standard estimates on the tail probabilities of binomial random variables (such as Lemma~\ref{lemma:Bennett-plus}) imply that with probability at least $1 - n^{-7}$, the maximum degree of $H_1$ is at most $2np$.

  \smallskip
  \noindent
  \textbf{Stage 2.}
  We shall describe an algorithm that with probability at least $1 - n^{-2}$ finds for each $s \in [N]$ an embedding $\varphi_s'$ of $S_s$ into the subgraph of $\Gnq$ induced by the set $W_s$ such that:
  \begin{itemize}
  \item
    the vertex $v_s$ is mapped to $\varphi(v_s)$, which was defined in Stage 1, and
  \item
    all $\varphi_1(L_1), \dotsc, \varphi_N(L_N)$ and $\varphi_1'(S_1), \dotsc, \varphi_N'(S_N)$ are pairwise edge-disjoint.
  \end{itemize}
  Clearly, this will complete the proof of the theorem.

  \smallskip
  \noindent
  \textbf{Algorithm.}
  Let $H_2$ be the empty graph with the same vertex set as $\Gnq$ and for each $s \in [N]$, do the following:
  \begin{enumerate}
  \item
    If the maximum degree of $H_2$ exceeds $np$, we abort the algorithm.
  \item
    Let $G^s$ be the subgraph of $K_n \setminus (H_1 \cup H_2)$ induced by the set $W_s$ and note that by~\eqref{eq:Ui-size},
    \[
    \delta(G^s) \ge |W_s| - \Delta(H_1) - \Delta(H_2) - 1 \ge |W_s| - 3np \ge (1 - 3p/\alpha)|W_s|.
    \]
    Moreover, observe that $G^s$ is disjoint from $\varphi_1(L_1), \dotsc, \varphi_N(L_N)$ and $\varphi_1'(S_1), \dotsc, \varphi_{s-1}'(S_{s-1})$.
  \item
    \label{item:main-step}
    Let $q_s = \Delta(\log n)^5 / |W_s|$. If there is an embedding $\varphi_s'$ of the $|W_s|$-vertex tree $S_s$ into an independent copy of the graph $G_{q_s}^s$ such that $\varphi_s'(v_s) = \varphi_s(v_s)$, then continue. Otherwise, abort the algorithm.
  \item
    Add to $H_2$ all the edges of $\varphi_s'(S_s)$.
  \end{enumerate}

  We first claim that the union $G^*$ of all $G_{q_s}^s$ is a subgraph of $\Gnq$. Indeed, since the graphs $G_{q_1}^1, \dotsc, G_{q_N}^N$ were independent, then for every pair of distinct vertices $x$ and $y$, recalling~\eqref{eq:Ui-distribution},
  \[
  \begin{split}
    \Pr\left(\{x,y\} \not\in G^* \right) & = \prod_{s=1}^N \left(1 - q_s \cdot \indicator[\{x,y\} \subseteq W_s] \right) \ge 1 - \sum_{s = 1}^N \indicator[\{x,y\} \subseteq W_s] \cdot q_s \\
    & = 1 - \sum_{s=1}^N \indicator[\{x,y\} \subseteq W_s] \cdot \frac{\Delta (\log n)^5}{|W_s|} \ge 1 - 4 \alpha p \cdot \Delta (\log n)^5 = 1-q.
  \end{split}
  \]
  independently of all other pairs. Second, we claim that the algorithm fails with probability at most $n^{-2}$. As at all times, $H_2 \subseteq G^* \subseteq \Gnq$ and $q = \eps p/2 \gg (\log n) / n$, standard estimates on the tail probabilities of binomial random variables (such as Lemma~\ref{lemma:Bennett-plus}) imply that with probability at least $1 - n^{-3}$, the maximum degree of $H_2$ is at most $np$. Moreover, as $3p/\alpha \le 1/(\Delta (\log n)^4)$ by our assumptions on $p$ and $\Delta$, Theorem~\ref{theorem:Montgomery} implies that the probability that the algorithm is aborted in step~\ref{item:main-step} of a given iteration of the main loop is at most $n^{-3}$. It follows that the algorithm succeeds with probability at least $1 - n^{-2}$.
\end{proof}

\section{Proof of Theorem \ref{theorem:main}}
\label{sec:proof of main}
Suppose that $\alpha$, $\varepsilon$, $\Delta$, $p$, and $n$ satisfy
\[
\eps \le 1/2, \qquad \frac{150 (\log n)^2}{\alpha \eps n} \le p \le \frac{\eps\alpha^4}{126}, \quad \text{and} \quad \Delta \le \min\left\{ \alpha, \frac{\eps}{\log(1/\alpha)} \right\} \cdot \frac{\eps n p}{1600 \log n}.
\]
Let $N \le (1-\eps)\frac{np}{2}$ and let $m = (1-\alpha)n$. Suppose that $T_1,\ldots, T_N$ is a collection of trees, each of which has at most $m$ vertices and maximum degree at most $\Delta$.

Our goal is to pack all the $T_i$ into $G_{n,p}$. In order to do so, we shall describe a randomised algorithm that tries to greedily construct a packing of $T_1, \ldots, T_N$ into the complete graph $K_n$ whose edges are labeled with elements of the interval $[0,1]$. We shall then prove that if the labels are independent uniform $[0,1]$-valued random variables, then with probability at least $1 - n^{-7}$ our algorithm constructs a packing of $T_1, \dotsc, T_N$ with the additional property that the labels of all the edges used by this packing do not exceed $p$. Denote the above event by $\Success$. As the subgraph comprising all edges whose labels fall into $[0,p]$ has the same distribution as $\Gnp$, we will be able to conclude that
\[
\Pr(\text{$T_1, \ldots, T_N$ pack into $\Gnp$}) \ge \Pr(\Success) \ge 1 - n^{-7}.
\]
Our embedding algorithm will try to embed the trees $T_1, \ldots, T_N$ one-by-one in $N$ consecutive rounds. During each round, it embeds the given tree $T_s$ vertex-by-vertex, while considering the vertices in the BFS ordering described in Section~\ref{sec:BFS}.

We find it illustrative to think that each edge $e$ of the complete graph is equipped with an alarm clock that will ring at (random) time $t_e$. The clock associated with $e$ shows time $c_e \in [0,1]$. At the beginning of the algorithm $c_e = 0$ for each $e$. The clocks will normally be stopped, but in each step of the algorithm, we will run a collection of them simultaneously until some number of them ring, that is, when $c_e$ reaches $t_e$ for a number of different $e$. All the edges whose clocks have just rung will be used in the embedding. We shall accomplish this by only running the clocks whose edges can be immediately used. Moreover, a clock that has rung permanently stops at $c_e = t_e$.

Let us fix an $s \in [N]$, let $T = T_s$, and let $m$ denote the number of vertices of $T$. (For the sake of brevity, we shall suppress the implicit index $s$ from our notation.) We let $v_0$ be an arbitrary vertex of $T$ and we root $T$ at $v_0$. We label the remaining vertices of $T$ as $v_1, \ldots, v_{m-1}$ according to the BFS ordering of $T$, which we defined in Section~\ref{sec:BFS}; we also let $J$ and $(d_i)_{i \in J}$ be as in Section~\ref{sec:BFS}.

We may now describe the embedding algorithm. Suppose that we have already embedded $T_1, \ldots, T_{s-1}$. For each edge $e$ of $K_n$, the clock associated with it shows some time $c_e \in [0,t_e]$. Moreover, $c_e = t_e$ if and only if $e \in \varphi_1(T_1) \cup \ldots \cup \varphi_{s-1}(T_{s-1})$. Let $v_0, \ldots, v_{m-1}$ be the ordering of the vertices of $T_s$ specified above. We map the root $v_0$ of $T_s$ to a uniformly chosen random vertex of $K_n$. Let $i \in J$ and suppose that $v_0, \ldots, v_{\ch{i}-1}$ have already been embedded. In particular, $v_i$ is already mapped to some vertex~$u$. We now try to embed the children of $v_i$, that is, $v_{\ch{i}}, \ldots, v_{\ch{i} + d_i-1}$. To this end, we shall run the clocks associated with all the edges $uw$ such that (i) $w$ has not yet been used in the embedding of $T_s$ and (ii) the clock associated with $uw$ still has not rung (i.e., $c_{uw} < t_{uw}$ or, equivalently, the edge $uw$ does not belong to $\varphi_1(T_1) \cup \ldots \cup \varphi_{s-1}(T_{s-1})$) until some $d_i$ of them ring. We map $v_{\ch{i}}, \ldots, v_{\ch{i} + d_i -1}$ to those $w$ for which the clock associated with $uw$ has just rung (in the exact same order as the $d_i$ clocks have just rung). We remark here that the clocks will be run at marginally different rates in order to assure that each of them has an equal chance of ringing.

We now give a formal description of the embedding algorithm. Denote the set of vertices of the host graph $K_n$ by $V$.

\smallskip
\noindent
\textbf{Algorithm.}
For each edge $e$ of $K_n$, define a new variable $c_e$ and set it to $0$. Moreover, let $t_e \in [0,1]$ be the (random) label of $e$. In each round $s = 1, \ldots, N$, do the following:

\begin{enumerate}
\item
  Let $T = T_s$ and let $v_0, \dotsc, v_{m-1}$ be the BFS ordering of the vertices of $T$ (rooted an an arbitrary vertex); let $J = J_s$ and $(d_i)_{i \in J}$ be as in Section~\ref{sec:BFS}.
\item
  Map $v_0$ to a uniformly chosen random vertex $u \in V$. In other words, let $\varphi = \varphi_s$ be the empty map and set $\varphi(v_0) = u$.
\item
  For each $i \in J$ do the following:
  \begin{enumerate}
  \item
    \label{item:step-beginning}
    Let $u \in V$ be the vertex where we have already mapped $v_i$, that is, $u = \varphi(v_i)$.
  \item
    \label{item:varphi-injective}
    Let $U_i = U_i^s \subseteq V$ be the set of vertices not yet used in the partial embedding of $T_s$, that is, $U_i = V \setminus \varphi(\{v_0, \ldots, v_{\ch{i}-1}\})$ and observe that $|U_i| = n-\ch{i}$.
  \item
    Define, for each $\tau \ge 0$,
    \[
    N_i(\tau) = N_i^s(\tau) = \{w \in U_i \colon c_{uw} < t_{uw} \le c_{uw} + (1-c_{uw}) \tau \}
    \]
    and note that $N_i(0) = \emptyset$ and $N_i(1) = \{w \in U_i \colon c_{uw} < t_{uw}\}$.
  \item
    Let us say that $w \in U_i$ \emph{enters} $N_i$ at time $\tau$ if $w \in N_i(\tau)$ but $w \not\in N_i(\tau')$ for all $\tau' < \tau$. (Observe that with probability one, no two vertices enter $N_i$ at the same time.)
  \item
    Let $\tau_i = \tau_{s,i}$ be the earliest time when $d_i$ vertices have entered $N_i$, that is,
    \[
    \tau_i = \min\{ \tau \ge 0 \colon |N_i(\tau)| \ge d_i\}.
    \]
    (Observe that with probability one, $|N_i(\tau_i)| = d_i$, provided that $|N_i(1)| \ge d_i$.)
  \item
    \label{item:step-children-images}
    Denote the $d_i$ vertices that have entered $N_i$ until $\tau_i$ by $u_1, \ldots, u_{d_i}$ (in this exact order). Map $v_{\ch{i}}, \ldots, v_{\ch{i}+d_i-1}$, which are the $d_i$~children of $v_i$ in $T$, to $u_1, \ldots, u_{d_i}$, respectively.
  \item
    \label{item:step-end}
    For every $w \in U_i$, update $c_{uw} \leftarrow \min\{t_{uw}, c_{uw} + (1-c_{uw})\tau_i\}$.
  \end{enumerate}
\item
  If the maximum degree of $\varphi_1(T_1) \cup \ldots \cup \varphi_s(T_s)$ exceeds $2np$, we terminate the algorithm.
\end{enumerate}

For every $s \in \{0, \ldots, N\}$, denote by $\DD_s$ the event that the maximum degree of the graph $\varphi_1(T_1) \cup \ldots \cup \varphi_s(T_s)$ does not exceed $2np$, so that $\DD_0$ holds always and for every $s \in [N]$, our algorithm terminates at the end of round $s$ if and only if $\DD_s$ does \emph{not} hold.

\begin{claim}
  \label{claim:Nisone-size}
  For every $s \in [N]$, if $\DD_{s-1}$ holds, then in the $s$th round of the algorithm,
  \[
  |N_i^s(1)| \ge n - \ch{i} - 2np \ge \alpha n - 2np
  \]
  for each $i \in J_s$. In particular, $|N_i^s(1)| \ge \Delta \ge \Delta(T_s) \ge d_i$.
\end{claim}
\begin{proof}
  Fix an $s \in [N]$, let $H = \varphi_1(T_1) \cup \ldots \cup \varphi_{s-1}(T_{s-1})$, fix an $i \in J_s$, and let $u = \varphi_s(v_i)$. Observe that for every $w \in U_i$, we have $t_{uw} \le c_{uw}$ (actually, $t_{uw} = c_{uw}$) precisely when $uw \in H$. In particular, $N_i(1)$ contains precisely those vertices $w \in U_i$ for which $uw \not\in H$. Therefore,
  \[
  |N_i(1)| \ge |U_i| - \Delta(H) = n - \ch{i} - \Delta(H).
  \]
  The claimed inequality follows as on the event $\DD_{s-1}$, the maximum degree of $H$ is at most $2np$.
\end{proof}

\begin{claim}
  If the algorithm has not terminated, it has constructed a packing of $T_1, \ldots, T_N$ into $K_n$. Moreover, the labels of the edges used in the packing do not exceed $\max_e c_e$.
\end{claim}
\begin{proof}
  The description of the algorithm guarantees that each $\varphi_s$ is an injection, see~\ref{item:varphi-injective}. In particular, $\varphi_s$ is an embedding of $T_s$ into $K_n$. More importantly, an edge $uv$ of $K_n$ is used in the embedding if and only if $t_{uw}$ belongs to one of the intervals $\big(c_{uw}, c_{uw} + (1-c_{uw}) \tau_i\big]$. This can happen only once during the entire execution of the algorithm as at the end of each round where $uw$ was considered, $c_{uw}$ is increased to either $t_{uw}$ or $c_{uw} + (1-c_{uw})\tau_i$. The second assertion follows as at the end of the execution of the algorithm, $c_e = t_e$ for every edge $e$ used in the embedding.
\end{proof}

Therefore, it will be sufficient to show that
\begin{equation}
  \label{eq:main-pr-estimate}
  \Pr\left(\comp{\DD_N} \vee \max_e c_e > p\right) \le n^{-7},
\end{equation}
which we shall do in the remainder of this section. For each $e \in K_n$ and $s \in [N]$, let $\tau_{e,s}$ denote the total time that the clock associated with $e$ was running during round $s$ of the algorithm, disregarding the rate at which the clock was running. As the rate is never more than one, one easily sees that $c_e \le \tau_{e,1} + \ldots + \tau_{e,N}$ for each $e$ at the end of the algorithm. With view of this, we shall be interested in bounding the probability that $\max_e \tau_{e,1} + \ldots + \tau_{e,N}$ exceeds $p$. Eventually, a sufficiently strong bound on this probability will follow from Lemma~\ref{lemma:Bennett-plus}. Unfortunately, as the distributions of the random variables $\tau_{e,s}$ seem difficult to describe explicitly, we shall first need some preparations.

Given an $s \in [N]$ and $i \in J_s$, we shall refer to the execution of~\ref{item:step-beginning}--\ref{item:step-end} during round $s$ for this particular $i$ as \emph{step} $(s,i)$. For every pair of distinct $u, w \in V$, every $s$ and $i$ as above, let
\[
E_{i,u,w} = E_{i,u,w}^s = \big\{\text{$\varphi_s(v_i) = u$ and $w \in N_i^s(1)$}\big\}.
\]
In particular, one of $E_{i,u,w}^s$ and $E_{i,w,u}^s$ holds if and only if the clock associated with $uw$ is running when we are trying to embed the children of $v_i$ in round $s$. It is now easy to convince oneself that
\begin{equation}
  \label{eq:tau-uw-s}
  \tau_{uw,s} = \sum_{i \in J_s} (\indicator[E_{i,u,w}^s] + \indicator[E_{i,w,u}^s]) \cdot \tau_{s,i}.
\end{equation}
Moreover, as the events $\bigcup_{i \in J_s} \{E_{i,u,w}^s, E_{i,w,u}^s\}$ are pairwise disjoint, we also have
\begin{equation}
  \label{eq:tau-uw-s-squared}
  \tau_{uw,s}^2 = \sum_{i \in J} (\indicator[E_{i,u,w}^s] + \indicator[E_{i,w,u}^s]) \cdot \tau_{s,i}^2.
\end{equation}

Given an $s \in [N]$, let $\FF_s$ denote the $\sigma$-algebra generated by what happened in the algorithm by the start of round $s$. Moroever, given an $i \in J_s$, let $\FF_{s,i}$ denote the $\sigma$-algebra generated by what happened in the algorithm by the start of step $(s,i)$, that is, right before the children of the vertex $v_i$ are embedded. The following two key lemmas will allow us to use the representations~\eqref{eq:tau-uw-s} and~\eqref{eq:tau-uw-s-squared} to bound the (conditional) expectations of $\tau_{e,s}$ and $\tau_{e,s}^2$ for all $s$ and $e$.

\begin{lemma}
  \label{lemma:ex-clock-running-time}
  For every $s$ and $i \in J_s$ and every positive integer $k$, letting $d = d_i$ and $r = |N_i^s(1)|$,
  \[
  \Ex\big[\tau_{s,i}^k \mid \FF_{s,i}\big] = \prod_{j=1}^{k} \frac{d+j-1}{r+j},
  \]
  provided that $d \le r$. Moreover, conditioned on $\FF_{s,i}$, the sequence $(u_1, \ldots, u_d)$ defined in~\ref{item:step-children-images} of step $(s,i)$ is a uniform random $d$-element ordered subset of $N_i^s(1)$.
\end{lemma}
\begin{proof}
  Observe first that conditioned on the clock at $e$ not having rung until $c_e$, the random variable $t_e$ is uniformly distributed on the interval $(c_e, 1]$. Therefore, conditioned on $\FF_{s,i}$, the variable $\tau_{s,i}$ has the same distribution as the $d$th smallest value among $r$ independent uniform $[0,1]$-valued random variables\footnote{This is often referred to as the \emph{$d$th order statistic}.}. Denote this random variable by $\tau$. The probability density function of $\tau$ is $t \mapsto d\binom{r}{d}t^{d-1}(1-t)^{r-d}$ and hence
  \[
  \begin{split}
    \Ex[\tau^k] &= d \binom{r}{d} \int_0^1 t^{d+k-1}(1-t)^{r-d} \, dt = d \binom{r}{d} B(d+k,r+1-d) \\
    & = \frac{r!}{(d-1)!(r-d)!} \cdot \frac{(d+k-1)!(r-d)!}{(r+k)!} = \prod_{j=1}^k \frac{d+j-1}{r+j},
  \end{split}
  \]
  where $B \colon \mathbb{Z}_+^2 \to \mathbb{R}$ is the Euler beta function, which is defined by
  \[
  B(x,y) = \int_0^1 t^{x-1} (1-t)^{y-1} \, dt = \frac{(x-1)!(y-1)!}{(x+y-1)!}.
  \]
  The second part of the lemma follows by symmetry.
\end{proof}

The second lemma, which is really the heart of the argument, provides upper bounds on the (conditional) probabilities of the events $E_{i,u,w}^s$ that appear in~\eqref{eq:tau-uw-s} and~\eqref{eq:tau-uw-s-squared}. Let $\delta = 21p / \alpha^4$ and note that
\begin{equation}
  \label{eq:exp-delta}
  e^\delta \ge 1+ \delta \ge 1+21p \ge \frac{n}{n-1} \qquad \text{and} \qquad e^{2\delta} \le 1 + 3\delta \le 1 + 63p/\alpha^4 \le 1+\eps/2.
\end{equation}

\begin{lemma}
  \label{lemma:uniform-dist}
  For every pair of distinct $u, w \in V$ and all $s \in [N]$ and $i \in J_s$, the following holds.
  \begin{equation}
    \label{eq:uniform-dist}
    \Pr\big( E_{i,u,w}^s \wedge \DD_{s-1} \mid \FF_s \big) \le \frac{n-\ch{i}}{n^2} \cdot e^\delta.
  \end{equation}
\end{lemma}

The proof of Lemma~\ref{lemma:uniform-dist} is quite technical and therefore we postpone it to the end of the section. Before proceeding with our proof, it would be useful to understand the intuition behind~\eqref{eq:uniform-dist}. By the description of our embedding algorithm, $\varphi(v_0)$ is a uniformly chosen random vertex in $V$. Moreover, by Lemma~\ref{lemma:ex-clock-running-time}, for every $j \in J$, conditioned on $\varphi_s(v_0), \ldots, \varphi_s(v_{\ch{j}-1})$, the images of the $d_j$ children of $v_j$ form a uniform random $d_j$-element ordered subset of $N_j^s(1)$. It follows that if $N_j^s(1) = U_j^s$ for every $j \in J$, then $\varphi_s(v_0), \ldots, \varphi_s(v_{m-1})$ would form a uniform random $m$-element ordered subset of~$V$. In particular, $\Pr\big(E_{i,u,w}^1 \mid \FF_1\big) = \frac{n-\ch{i}}{n(n-1)} \le \frac{n-\ch{i}}{n^2} \cdot e^\delta$.

Unfortunately, this is true only if $s = 1$, as in reality $N_j^s(1)$ contains only those vertices $w$ of $U_j^s$ for which the edge $\{\varphi_s(v_j),w\}$ has not already appeared in $H = \varphi_1(T_1) \cup \ldots \cup \varphi_{s-1}(T_{s-1})$. Clearly, $|U_j^s \setminus N_j^s(1)| \le \Delta(H)$ for every $j$ and hence one would expect that if $\Delta(H)$ is not too large, then the distribution of $\varphi_s(v_0), \ldots, \varphi_s(v_{m-1})$ is not very far from uniform. The content of~\eqref{eq:uniform-dist} is that the above intuition is indeed true (in some precise quantitative sense).

For every $s \in [N]$ and every $e \in E(K_n)$, let $c_e^s$ denote the value of $c_e$ at the end of the $s$th round of the algorithm (so that $c_e = c_e^N$). Similarly as before, one easily sees that $c_e^s \le \tau_{e,1} + \ldots + \tau_{e,s}$. As $\DD_0$ holds always, our main probabilistic estimate, inequality~\eqref{eq:main-pr-estimate}, will easily follow from the following statement.

\begin{lemma}
  \label{lemma:main-pr-estimate-step}
  For every $\smax \in [N]$, the following holds:
  \begin{equation}
    \label{eq:main-pr-estimate-step}
    \Pr\left(\DD_{\smax-1} \text{ and } \left(\max_e c_e^\smax > p \text{ or }  \comp{\DD_\smax}\right)\right) \le 3n^{-8}.
  \end{equation}
\end{lemma}

Indeed, as the event $\DD_0$ holds trivially, we have $\comp{\DD_N} = \bigcup_{s=1}^N \comp{\DD_s} \cap \DD_{s-1}$ and hence by Lemma~\ref{lemma:main-pr-estimate-step},
\begin{equation}
  \label{eq:Pr-DD-N}
  \Pr\left( \comp{\DD_N} \right) \le \sum_{s=1}^N \Pr\left( \DD_{s-1} \cap \comp{\DD_s} \right) \le N \cdot 3n^{-8}
\end{equation}
Since $\DD_N$ clearly implies $\DD_{N-1}$ and $c_e = c_e^N$, then again by Lemma~\ref{lemma:main-pr-estimate-step},
\begin{equation}
  \label{eq:Pr-max-ce}
  \Pr\left( \DD_N \wedge \max_e c_e > p\right) \le \Pr\left( \DD_{N-1} \wedge \max_e c_e^N > p\right) \le 3n^{-8}.
\end{equation}
Finally, \eqref{eq:Pr-DD-N} and~\eqref{eq:Pr-max-ce} immediately give~\eqref{eq:main-pr-estimate}.

\begin{proof}[{Proof of Lemma~\ref{lemma:main-pr-estimate-step}}]
  Fix an $\smax \in [N]$ and an $e \in E(K_n)$ and recall that $c_e^\smax \le \tau_{e,1} + \ldots + \tau_{e,\smax}$. As a preparation to invoke Lemma~\ref{lemma:Bennett-plus}, we first estimate, for each $s \in [\smax]$ and each pair $uw \in E(K_n)$, the conditional expectations of $\tau_{uw,s}$ and $\tau_{uw,s}^2$ given $\FF_s$ (on the event $\DD_{s-1} \supseteq \DD_{\smax-1}$). To this end, recall first that (i) the events $E_{i,u,w}^s$, $E_{i,w,u}^s$, and $\DD_{s-1}$ are all in $\FF_{s,i}$ and (ii) on the event $\DD_{s-1}$, the set $N_i^s(1)$ has at least $n - \ch{i} - 2np$ elements, see Claim~\ref{claim:Nisone-size}. It now follows from~\eqref{eq:tau-uw-s} and Lemmas~\ref{lemma:ex-clock-running-time} and~\ref{lemma:uniform-dist} that
\[
\begin{split}
  \Ex\left[ \tau_{uw,s} \cdot \indicator[\DD_{s-1}] \mid \FF_s \right] & = \sum_{i \in J_s} \Ex\left[ \left( \indicator[E_{i,u,w}^s] + \indicator[E_{i,w,u}^s] \right) \cdot \tau_{s,i} \cdot \indicator[\DD_{s-1}] \mid \FF_s \right] \\
    & = \sum_{i \in J_s} \Ex\left[ \left( \indicator[E_{i,u,w}^s] + \indicator[E_{i,w,u}^s] \right) \cdot \Ex\left[ \tau_{s,i} \mid \FF_{s,i} \right] \cdot \indicator[\DD_{s-1}] \mid \FF_s \right] \\
    & \le \sum_{i \in J_s} \frac{d_i}{n-\ch{i}-2np} \cdot \left( \Pr\left( E_{i,u,w}^s \wedge \DD_{s-1} \mid \FF_s \right) + \Pr\left( E_{i,w,u}^s \wedge \DD_{s-1} \mid \FF_s \right) \right) \\
    & \le \sum_{i \in J_s} \frac{d_i}{n-\ch{i}-2np} \cdot \frac{n-\ch{i}}{n^2} \cdot 2e^\delta \le \frac{2e^\delta}{n^2} \cdot \frac{\alpha n}{\alpha n - 2np} \cdot \sum_{i \in J_s} d_i \\
    & = \frac{2e^\delta}{n^2} \cdot \frac{\alpha}{\alpha - 2p} \cdot |E(T_s)| \le \frac{2}{n} \cdot e^{\delta + \frac{2p}{\alpha - 2p}} \le \frac{2}{n} \cdot e^{2\delta} \le \frac{2+\eps}{n},
\end{split}
\]
where the final inequality is~\eqref{eq:exp-delta}. In a similar fashion, it follows from~\eqref{eq:tau-uw-s-squared} that
\[
\begin{split}
  \Ex\left[ \tau_{uw,s}^2 \cdot \indicator[\DD_{s-1}] \mid \FF_s \right] & \le \sum_{i \in J_s} \frac{d_i(d_i+1)}{(n-\ch{i}-2np)^2} \cdot \frac{n-\ch{i}}{n^2} \cdot 2e^\delta \le \frac{4e^\delta}{n^2} \cdot \sum_{i \in J_s} \frac{d_i^2}{n-\ch{i}} \cdot \left(\frac{\alpha}{\alpha-2p}\right)^2 \\
  & \le \frac{4e^{2\delta}}{n^2} \cdot \sum_{i \in J_s} \frac{d_i^2}{n-\ch{i}} \le \frac{5}{n^2} \cdot \Delta(T_s) \cdot \sum_{i \in J_s} \frac{d_i}{n-\ch{i}} \le \frac{5 \Delta \log (1/\alpha)}{n^2},
\end{split}
\]
where the final inequality is~\eqref{eq:sum-degvk-nkk}. Furthermore, let
\[
\taumax = \frac{\eps p}{60\log n}
\qquad \text{and} \qquad
\taunorm_{uw,s} = \min\{ \tau_{uw,s}, \taumax\} \cdot \indicator[\DD_{s-1}].
\]
Observe that for each pair $uw \in E(K_n)$, the random variables $\taunorm_{uw,1}, \ldots, \taunorm_{uw,\smax}$ satisfy the assumptions of Lemma~\ref{lemma:Bennett-plus} with
\[
M \leftarrow \taumax, \qquad \mu \leftarrow \frac{2+\eps}{n}, \qquad \text{and} \qquad \sigma^2 \leftarrow \frac{5\Delta \log (1/\alpha)}{n^2}.
\]
In particular, letting $t = \eps p / 4$, we see that
\[
p - \smax\mu \ge p - N\mu \ge p - (1-\eps) \frac{np}{2} \cdot \frac{2+\eps}{n} = \left(1 - \frac{(1-\eps)(2+\eps)}{2}\right) p = \frac{\eps - \eps^2}{2}p \ge t
\]
and also
\[
2 s' \sigma^2 \le 2 N \sigma^2 \le np \cdot \frac{5 \Delta \log (1/\alpha)}{n^2} = \frac{5p \Delta \log (1/\alpha)}{n} \le \frac{\eps^2 p^2}{320 \log n} \qquad \text{and} \qquad \frac{2M t}{3} \le \frac{\eps^2 p^2}{320\log n}.
\]
Therefore, it follows from Lemma~\ref{lemma:Bennett-plus} that
\begin{equation}
  \label{eq:Pr-sum-taunorm}
  \Pr\left(\taunorm_{uw,1} + \ldots + \taunorm_{uw,\smax} > p\right) \le \exp\left(- \frac{t^2}{2(\smax\sigma^2 + Mt/3)} \right) \le n^{-10}.
\end{equation}

Now, observe that if $\tau_{e,s} = \taunorm_{e,s}$ for all $s \in [\smax]$ and $e \in E(K_n)$, then $c_e^\smax \le \taunorm_{e,1} + \ldots + \taunorm_{e,\smax}$. Since each $\tau_{e,s}$ equals either zero or $\tau_{s,i}$ for some $i \in J_s$, the former event holds precisely when $\tau_{s,i} \le \taumax$ for all $s \in [\smax]$ and $i \in J_s$. In particular, it follows from~\eqref{eq:Pr-sum-taunorm}, \eqref{eq:Pr-sum-taunorm}, and the union bound that
\begin{equation}
  \label{eq:Pr-ce-large-on-DDs}
  \Pr\left( \DD_{\smax-1} \wedge \max_e c_e^\smax > p \right) \le n^{-8} + \sum_{s =1}^{\smax} \sum_{i \in J_s} \Pr( \DD_{\smax-1} \wedge \tau_{s,i} > \taumax).
\end{equation}
In order to estimate the right-hand side of~\eqref{eq:Pr-ce-large-on-DDs}, note that $\DD_{\smax-1} \subseteq \DD_{s-1}$ for every $s \in [\smax]$ and hence by Claim~\ref{claim:Nisone-size}, on $\DD_{\smax-1}$, the set $N_i^s(1)$ has at least $\alpha n - 2np$ elements. Therefore, by Lemma~\ref{lemma:ex-clock-running-time}, for every $s$ and $i$ as above and every positive integer $k$, using Markov's inequality,
\begin{equation}
  \label{eq:Pr-tau-taumax}
  \begin{split}
    \Pr\left(\DD_{\smax-1} \wedge \tau_{s,i} > \taumax \right) & \le \Pr\left(\DD_{s-1} \wedge \tau_{s,i}^k > \taumax^k \right) \le \taumax^{-k} \cdot \Ex\left[\tau_{s,i}^k \cdot \indicator[\DD_{s-1}]\right] \\
    & \le \taumax^{-k} \prod_{j=1}^k \frac{d_i+j-1}{\alpha n - 2np + j} \le \left( \frac{\Delta+k}{\taumax \cdot (\alpha-2p)n} \right)^k \\
    & \le \left(\frac{120(\Delta+k)\log n}{\eps \alpha n p}\right)^k \le \left(\frac{1}{4} + \frac{k \log n}{\eps \alpha n p}\right)^k.
  \end{split}
\end{equation}
Substituting $k = \lceil 10 \log_2 n \rceil \le \eps \alpha np / (4\log n)$ into~\eqref{eq:Pr-tau-taumax} yields
\[
\Pr\left(\DD_{\smax-1} \wedge \tau_{s,i} > \taumax \right) \le n^{-10},
\]
which together with~\eqref{eq:Pr-ce-large-on-DDs} gives
\begin{equation}
  \label{eq:Pr-ce-large-on-DDs-final}
    \Pr\left( \DD_{\smax-1} \wedge \max_e c_e^\smax > p \right) \le n^{-8} + \smax \cdot n \cdot n^{-10} \le 2n^{-8}.
\end{equation}

Finally, we estimate the probability of $\DD_{\smax-1} \wedge \max_e c_e^\smax \le p \wedge \comp{\DD_{\smax}}$. To this end, note that the graph $\varphi_1(T_1) \cup \ldots \cup \varphi_\smax(T_\smax)$ is contained in the graph comprising all edges $f$ with $t_f \le \max_e c_e^\smax$. In particular,
\begin{equation}
  \label{eq:Pr-ce-small-not-DDs}
  \Pr\left(\comp{\DD_{\smax}} \wedge \max_e c_e^\smax \le p \right) \le \Pr\left( \Delta(G_{n,p}) > 2np \right) \le n \cdot \exp\left(- np / 4\right),
\end{equation}
where the last inequality is a standard estimate for the upper tail of the binomial distribution that can be easily derived using Lemma~\ref{lemma:Bennett-plus}. As $np/4 \ge 9\log n$, inequalities~\eqref{eq:Pr-ce-large-on-DDs-final} and~\eqref{eq:Pr-ce-small-not-DDs} immediately yield~\eqref{eq:main-pr-estimate-step}.
\end{proof}

In order to complete the proof we shall now prove Lemma ~\ref{lemma:uniform-dist}.

\begin{proof}[{Proof of Lemma~\ref{lemma:uniform-dist}}]
  We first handle the easy case $i = 0$. Since conditioned on $\FF_s$, the vertex $\varphi_s(v_0)$ is chosen uniformly at random from $V$, then
  \[
  \Pr(E_{0,u,w}^s \mid \FF_s) \le \Pr(\varphi_s(v_0) = u\mid \FF_s) = \frac{1}{n} \le \frac{n-1}{n^2} \cdot e^\delta,
  \]
  where the last inequality follows form~\eqref{eq:exp-delta}. Therefore, for the remainder of the proof, we shall assume that $i > 0$. Let $j_i$ be the index of the parent of $v_i$ in $T_s$, so that $\ch{j_i} \le i < \ch{j_i} + d_{j_i}$. Next, let $H = \varphi_1(T_1) \cup \ldots \cup \varphi_{s-1}(T_{s-1})$, let $B = N_H(u) \cup N_H(w)$, and note that $|B| \le 2\Delta(H)$. Without loss of generality, we may assume that (i) $u, w \not\in B$, as otherwise $uw \in H$ and consequently $\Pr(E_{i,u,w}^s \mid \FF_s) = 0$, and that (ii) $\Delta(H) \le D = 2np$ holds always, as otherwise the left-hand side of~\eqref{eq:uniform-dist} is zero.

  Let $A_{-1}$ denote the event that $\varphi_s(v_0) \not\in \{u,w\}$ and for every $j \in J \setminus \{j_i\}$, let $A_j$ denote the event that $u,w \not\in \{\varphi_s(v_{\ch{j}}), \ldots, \varphi_s(v_{\ch{j}+d_j-1}) \}$, that is, $u$ and $w$ are not among the images of the $d_j$ children of $v_j$ in $T_s$. Finally, denote by $A_{j_i}$ the event that $\varphi(v_i) = u$ and $w \not\in \{\varphi_s(v_{\ch{j_i}}), \ldots, \varphi_s(v_{\ch{j_i}+d_{j_i}-1}) \}$. Observe that the event $E_{i,u,w}^s$ can be expressed as an intersection of a sequence of events $A_j$, namely
  \begin{equation}
    \label{eq:Eiuw-as-cap-Aj}
    E_{i,u,w}^s = A_{-1} \cap \bigcap_{j \in J, j < i} A_j.
  \end{equation}
  With foresight, for every $j \in J$, define
  \begin{equation}
    \label{eq:Pj-def}
    P_j =
    \begin{cases}
      \left(1 - \frac{d_j}{n-\ch{j}}\right)^2, & \text{if $j < j_i$}, \\
      \frac{e^{2D/(\alpha n)}}{n-\ch{j}} \cdot \left(1 - \frac{d_j}{n-\ch{j}}\right), & \text{if $j = j_i$}, \\
      1 - \frac{d_j}{n-\ch{j}}, & \text{if $j > j_i$}.
    \end{cases}
  \end{equation}
  A good way to digest~\eqref{eq:Pj-def} is to observe the following. If $\varphi_s(v_0), \ldots, \varphi_s(v_{m-1})$ formed a uniform random $m$-element ordered subset of $V$, then $P_k$ would be (approximately) equal to the conditional probability of the event $A_k$ occurring, conditioned on $\bigcap_{j < k} A_j$ occurring. In particular, we have the following identity:
  \begin{equation}
    \label{eq:prod-Pj}
    \prod_{j \in J, j < i} P_j = e^{2D/(\alpha n)} \cdot \frac{n-\ch{i}}{(n-1)^2}.
  \end{equation}
  To see~\eqref{eq:prod-Pj}, note first that $\ch{j} + d_j = \ch{\suc{j}}$ for every $j \in J$ (recall that $\suc{j}$ is the successor of $j$ in $J$) and hence for every $k \in J$,
  \[
  \prod_{j \in J, j < k} \left(1 - \frac{d_j}{n-\ch{j}}\right) = \prod_{j \in J, j < k} \frac{n-\ch{\suc{j}}}{n-\ch{j}} = \frac{n-\ch{k}}{n-\ch{0}} = \frac{n-\ch{k}}{n-1}.
  \]
  Denoting the left-hand side of~\eqref{eq:prod-Pj} by $P$, we now see that
  \[
  P = \frac{e^{2D/(\alpha n)}}{n-\ch{j_i}} \cdot \prod_{j \in J, j < j_i} \left(1 - \frac{d_j}{n-\ch{j}}\right) \cdot \prod_{j \in J, j < i} \left(1 - \frac{d_j}{n-\ch{j}}\right) = e^{2D/(\alpha n)} \cdot \frac{n-\ch{i}}{(n-1)^2}.
  \]
  Therefore, in order to bound $\Pr(E_{i,u,w}^s \mid \FF_s)$, it will be enough to bound the ``conditional probability $\Pr(A_k \mid \bigcap_{j < k} A_j)$'' from above by $P_k$ (times a small error term) for each $k \le i$ and then use the chain rule for conditional probabilities.

  We now formalise the above discussion. If $j \in J \setminus \{j_i\}$, then by Lemma~\ref{lemma:ex-clock-running-time},
  \begin{equation}
    \label{eq:Pr-Aj-FFsj}
    \begin{split}
      \Pr(A_j \mid \FF_{s,j}) & = \left(1 - \frac{d_j}{|N_j^s(1)|}\right)^{\indicator[u \in N_j^s(1)]} \left(1 - \frac{d_j}{|N_j^s(1) \setminus \{u\}|}\right)^{\indicator[w \in N_j^s(1)]} \\
      & \le \left(1 - \frac{d_j}{|U_j^s|}\right)^{\indicator[u \in N_j^s(1)] + \indicator[w \in N_j^s(1)]}.
    \end{split}
  \end{equation}
  where the inequality holds as $N_j^s(1) \subseteq U_j^s$. Recall (e.g., from the proof of Claim~\ref{claim:Nisone-size}) that an $x \in V$ belongs to $N_j^s(1)$ if and only if $x \in U_j^s$ and $\varphi_s(v_j) \not\in N_H(x)$ and hence,
  \begin{equation}
    \label{eq:Pr-Aj-FFsj-indicators}
    \indicator[u \in N_j^s(1)] +   \indicator[w \in N_j^s(1)] \ge \indicator[u \in U_j^s] + \indicator[w \in U_j^s] - 2 \cdot \indicator[\varphi_s(v_j) \in B].
  \end{equation}
  Putting \eqref{eq:Pr-Aj-FFsj} and \eqref{eq:Pr-Aj-FFsj-indicators} together yields, recalling that $|U_j^s| = n - \ch{j}$,
  \begin{equation}
    \label{eq:Pr-Aj-FFsj-continued}
    \Pr(A_j \mid \FF_{s,j}) \le \left(1 - \frac{d_j}{n-\ch{j}}\right)^{\indicator[u \in U_j^s] + \indicator[w \in U_j^s] -2 \cdot \indicator[\varphi_s(v_j) \in B]}.
  \end{equation}
  As $d_j \le \Delta \le \alpha n / 1600$ and $n - \ch{j} \ge n - m \ge \alpha n$, we can estimate
  \begin{equation}
    \label{eq:Pr-Aj-FFsj-error-term}
    \left(1 - \frac{d_j}{n-\ch{j}}\right)^{-1} \le \exp\left(\frac{3d_j}{2(n - \ch{j})} \right).
  \end{equation}
  Substituting \eqref{eq:Pr-Aj-FFsj-error-term} into \eqref{eq:Pr-Aj-FFsj-continued}, we obtain
  \begin{equation}
    \label{eq:Pr-Aj-FFsj-bound}
    \Pr(A_j \mid \FF_{s,j}) \le \left(1 - \frac{d_j}{n-\ch{j}}\right)^{\indicator[u \in U_j^s] + \indicator[w \in U_j^s]} \cdot \exp\left( 3 \cdot \indicator[\varphi_s(v_j) \in B] \cdot \frac{d_j}{n-\ch{j}} \right).
  \end{equation}

  In the remaining case $j = j_i$, Lemma~\ref{lemma:ex-clock-running-time} implies that
  \begin{equation}
    \label{eq:Pr-Aji-FFsji}
    \Pr(A_{j_i} \mid \FF_{s,j_i}) = \left(1 - \frac{d_{j_i}}{|N_{j_i}^s(1)|}\right)^{\indicator[w \in N_{j_i}^s(1)]} \cdot \frac{\indicator[u \in N_{j_i}^s(1)]}{|N_{j_i}^s(1) \setminus \{w\}|}.
  \end{equation}
  Similarly as above, the first term in the right-hand side of~\eqref{eq:Pr-Aji-FFsji} may be estimated as follows:
  \begin{equation}
    \label{eq:Pr-Aji-FFsji-first-term}
    \begin{split}
      \left(1 - \frac{d_{j_i}}{|N_{j_i}^s(1)|}\right)^{\indicator[w \in N_{j_i}^s(1)]} & \le \left(1 - \frac{d_{j_i}}{n - \ch{j_i}}\right)^{\indicator[w \in U_{j_i}^s] - \indicator[\varphi_s(v_{j_i}) \in B]} \\
      & \le \left(1 - \frac{d_{j_i}}{n-\ch{j_i}}\right)^{\indicator[w \in U_{j_i}^s]} \cdot \exp\left( 3 \cdot \indicator[\varphi_s(v_{j_i}) \in B] \cdot \frac{d_{j_i}}{n-\ch{j_i}} \right).
    \end{split}
  \end{equation}
  To estimate the second term, we may use Claim~\ref{claim:Nisone-size} and the inequality $n - \ch{j_i} \ge n - m +1 \ge \alpha n + 1$:
  \begin{equation}
    \label{eq:Pr-Aji-FFsji-second-term}
    \frac{\indicator[u \in N_{j_i}^s(1)]}{|N_{j_i}^s(1) \setminus \{w\}|} \le \frac{1}{n - \ch{j_i} - D - 1} \le \frac{\alpha n + 1}{\alpha n - D} \cdot \frac{1}{n-\ch{j_i}} \le \frac{e^{2D/(\alpha n)}}{n-\ch{j_i}},
  \end{equation}
  where the last inequality holds as $D = 2np \le \alpha n / 63$.
  Putting \eqref{eq:Pr-Aji-FFsji}, \eqref{eq:Pr-Aji-FFsji-first-term}, and \eqref{eq:Pr-Aji-FFsji-second-term} together yields
  \begin{equation}
    \label{eq:Pr-Aji-FFsji-bound}
    \Pr(A_{j_i} \mid \FF_{s,j_i}) \le \frac{e^{2D/(\alpha n)}}{n-\ch{j_i}} \cdot \left(1 - \frac{d_{j_i}}{n-\ch{j_i}}\right)^{\indicator[w \in U_{j_i}^s]} \cdot \exp\left( 3 \cdot \indicator[\varphi_s(v_{j_i}) \in B] \cdot \frac{d_{j_i}}{n-\ch{j_i}} \right).
  \end{equation}

  If the set $B$ was empty, the somewhat annoying exponential error terms involving $\indicator[\varphi_s(v_j) \in B]$ would disappear from both \eqref{eq:Pr-Aj-FFsj-bound} and~\eqref{eq:Pr-Aji-FFsji-bound} and one could easily derive the claimed upper bound on the probability of $E_{i,u,w}^s$ arguing similarly as in the proof of~\eqref{eq:prod-Pj}. Unfortunately this in true only if $s=1$ and the treatment of the general case ($B \neq \emptyset$), which is the main business of this lemma, requires considerable effort.

  First, let us define, for every $I \subseteq J$,
  \[
  X_I = \exp\left( 3 \sum_{k \in I} \indicator[\varphi_s(v_k) \in B] \cdot \frac{d_k}{n-\ch{k}} \right),
  \]
  so that the exponential terms in the right-hand sides of~\eqref{eq:Pr-Aj-FFsj-bound} and~\eqref{eq:Pr-Aji-FFsji-bound} are simply $X_{\{j\}}$ and $X_{\{j_i\}}$, respectively. The following estimate is key.

  \begin{claim}
    \label{claim:Ex-Aj-XI}
    For every $j \in J$ and $I \subseteq \{\ch{j}, \dotsc, \ch{j} + d_j - 1\}$, the followings holds.
    \begin{enumerate}[label={\rm(\textit{\roman*})}]
    \item
      \label{item:Ex-Aj-XI}
      If $j \neq j_i$, then
      \[
      \Ex\big[\indicator[A_j] \cdot X_I \mid \FF_{s,j} \big] \le \left(1 - \frac{d_j}{n-\ch{j}}\right)^{\indicator[u \in U_j^s] + \indicator[w \in U_j^s]} \cdot X_{\{j\}} \cdot \exp\left(\frac{8D}{\alpha^3} \cdot \sum_{k \in I} \frac{d_k}{(n-\ch{k})^2} \right).
      \]
    \item
      \label{item:Ex-AjiXI}
      If $j = j_i$, then
      \[
      \Ex\big[\indicator[A_{j_i}] \cdot X_I \mid \FF_{s,j} \big] \le \frac{e^{2D/(\alpha n)}}{n-\ch{j_i}} \cdot \left(1 - \frac{d_{j_i}}{n-\ch{j_i}}\right)^{\indicator[w \in U_{j_i}^s]} \cdot X_{\{j_i\}} \cdot \exp\left(\frac{8D}{\alpha^3} \cdot \sum_{k \in I} \frac{d_k}{(n-\ch{k})^2} \right).
      \]
    \end{enumerate}
  \end{claim}

  In order to prove Claim~\ref{claim:Ex-Aj-XI}, we first argue that for all $I \subseteq J$,
  \begin{equation}
    \label{eq:XI-linear-bound}
    X_I \le 1 + \alpha^{-3} \cdot \sum_{k \in I} \indicator[\varphi_s(v_k) \in B] \cdot \frac{d_k}{n-\ch{k}}.
  \end{equation}
  Indeed, \eqref{eq:XI-linear-bound} follows from the fact that $e^x \le 1+e^a \cdot x$ for all $x \in [0,a]$ and the inequality
  \begin{equation}
    \label{eq:sum-degvk-nkk}
    \begin{split}
      \sum_{k \in I} \frac{d_k}{n-\ch{k}} & \le \sum_{k \in J} \frac{d_k}{n-\ch{k}} \le \sum_{k \in J} \sum_{d = 0}^{d_k-1} \frac{1}{n-\ch{k}-d} = \sum_{d = 1}^{m-1} \frac{1}{n-d} \\
      & = H_{n-1} - H_{n-m} \le \log \frac{n-1}{n-m} \le \log \frac{1}{\alpha},
    \end{split}
  \end{equation}
  where $H_d = \sum_{i=1}^d \frac{1}{d}$ is the $d$th harmonic number and we used the well-known fact that $d \mapsto H_d - \log d$ is monotonically decreasing.

  Fix a $j \in J$ and assume that $I \subseteq \{\ch{j}, \ldots, \ch{j} + d_j-1\}$. By Lemma~\ref{lemma:ex-clock-running-time}, conditioned on $\FF_{s,j}$, each $\varphi_s(v_k)$ with $k \in I$ is a uniformly chosen random element of the set $N_j^s(1)$ and hence~\eqref{eq:XI-linear-bound} yields
  \begin{equation}
    \label{eq:XI-minus-one-bound}
    \begin{split}
      \Ex[X_I - 1 \mid \FF_{s,j}] & \le \alpha^{-3} \cdot \sum_{k \in I} \frac{|B| }{|N_j^s(1)|} \cdot \frac{d_k}{n-\ch{k}} \\
      & \le \alpha^{-3} \cdot \sum_{k \in I} \frac{2D}{n-\ch{j}-D} \cdot \frac{d_k}{n-\ch{k}} \le \frac{4D}{\alpha^3} \cdot \sum_{k \in I} \frac{d_k}{(n-\ch{k})^2},
    \end{split}
  \end{equation}
  where in the second and the third inequalities we used Claim~\ref{claim:Nisone-size} and the inequalities $\ch{k} > \ch{j}$ and $n - \ch{j} - D \ge  (n-\ch{j})/2$ (which follows as $2D \le \alpha n \le n - \ch{j}$), respectively.

  Now, given a $j \in J \setminus \{j_i\}$ and an $I$ as above, we estimate the conditional expectation of $\indicator[A_j] \cdot X_I$, conditioned on $\FF_{s,j}$. To this end, note first that $X_I \ge 1$ and hence $\indicator[A_j] \cdot X_I \le \indicator[A_j] + X_I - 1$. In particular, we may invoke~\eqref{eq:Pr-Aj-FFsj-bound} and~\eqref{eq:XI-minus-one-bound} directly to obtain
  \begin{equation}
    \label{eq:Ex-Aj-XI}
    \Ex\big[\indicator[A_j] \cdot X_I \mid \FF_{s,j} \big] \le \left(1 - \frac{d_j}{n-\ch{j}}\right)^{\indicator[u \in U_j^s] + \indicator[w \in U_j^s]} \cdot X_{\{j\}} + \frac{4D}{\alpha^3} \cdot \sum_{k \in I} \frac{d_k}{(n-\ch{k})^2}.
  \end{equation}
  Now, item~\ref{item:Ex-Aj-XI} of Claim~\ref{claim:Ex-Aj-XI} is a straightforward consequence of~\eqref{eq:Ex-Aj-XI}, the simple estimate
  \begin{equation}
    \label{eq:Aj-lower-bound}
    \left(1 - \frac{d_j}{n-\ch{j}}\right)^{\indicator[u \in U_j^s] + \indicator[w \in U_j^s]} \cdot X_{\{j\}} \ge \left(1 - \frac{d_j}{n-\ch{j}}\right)^2 \ge \left(1 - \frac{\Delta}{\alpha n}\right)^2 \ge \frac{1}{2},
  \end{equation}
  and the inequality $1+x \le e^x$.

  We now estimate the conditional expectation of $\indicator[A_{j_i}] \cdot X_I$. Unfortunately, a bound akin to~\eqref{eq:Aj-lower-bound} does not hold for $\Pr(A_{j_i} \mid \FF_{s,j})$ and hence in order to obtain a suitable upper bound for $\Ex\big[\indicator[A_{j_i}] \cdot X_I\big]$, we need to argue somewhat differently, reiterating some of the above computations. As $A_{j_i}$ implies that $\varphi_s(v_i) = u$, it follows from~\eqref{eq:XI-linear-bound} that
  \begin{equation}
    \label{eq:Aji-XI-linear-bound}
    \indicator[A_{j_i}] \cdot (X_I - 1) \le \alpha^{-3} \cdot \sum_{k \in I} \indicator[\varphi_s(v_i) = u \wedge \varphi_s(v_k) \in B] \cdot \frac{d_k}{n-\ch{k}}.
  \end{equation}
  By Lemma~\ref{lemma:ex-clock-running-time}, conditioned on $\FF_{s,j_i}$, for each $k \in I \setminus \{i\}$, the pair $\big(\varphi_s(v_i), \varphi_s(v_k)\big)$ is a uniformly chosen random $2$-element ordered subset of $N_{j_i}^s(1)$. In particular, by~\eqref{eq:Pr-Aji-FFsji-second-term},
  \begin{equation}
    \label{eq:v-ell-v-i}
    \Pr\big(\varphi_s(v_i) = u \wedge \varphi_s(v_k) \in B \mid \FF_{s,j_i} \big) \le \frac{1}{|N_{j_i}^s(1)|} \cdot \frac{|B|}{|N_{j_i}^s(1)|-1} \le \frac{e^{2D/(\alpha n)}}{n-\ch{j_i}} \cdot \frac{|B|}{|N_{j_i}^s(1)|},
  \end{equation}
  holds for each $k$. Indeed, when $k = i$, then~\eqref{eq:v-ell-v-i} holds trivially, as we assumed that $u \not\in B$ and hence the left-hand side of~\eqref{eq:v-ell-v-i} is zero. Substituting~\eqref{eq:v-ell-v-i} into~\eqref{eq:Aji-XI-linear-bound} and using~\eqref{eq:XI-minus-one-bound}, we obtain
  \begin{equation}
    \label{eq:Aji-XI-minus-one-bound}
    \Ex\big[ \indicator[A_{j_i}] \cdot (X_I - 1) \mid \FF_{s,j} \big] \le \frac{e^{2D/(\alpha n)}}{n-\ch{j_i}} \cdot \frac{4D}{\alpha^3} \cdot \sum_{k \in I} \frac{d_k}{(n-\ch{k})^2}.
  \end{equation}
  Finally, as $\indicator[A_{j_i}] \cdot X_I  = \indicator[A_{j_i}] \cdot (X_I - 1) + \indicator[A_{j_i}]$, combining~\eqref{eq:Pr-Aji-FFsji-bound} with~\eqref{eq:Aji-XI-minus-one-bound} with yields
  \begin{equation}
    \label{eq:Ex-Aji-XI}
    \Ex\big[\indicator[A_{j_i}] \cdot X_I \mid \FF_{s,j} \big] \le \frac{e^{2D/(\alpha n)}}{n-\ch{j_i}} \cdot \left[\left(1 - \frac{d_{j_i}}{n-\ch{j_i}}\right)^{\indicator[w \in U_{j_i}^s]} \cdot X_{\{j_i\}} + \frac{4D}{\alpha^3} \cdot \sum_{k \in I} \frac{d_k}{(n-\ch{k})^2}\right].
  \end{equation}
  Now, item~\ref{item:Ex-AjiXI} of Claim~\ref{claim:Ex-Aj-XI} is a simple consequence of~\eqref{eq:Ex-Aji-XI}, the simple estimate
  \[
  \left(1-\frac{d_{j_i}}{n-\ch{j_i}}\right)^{\indicator[w \in U_{j_i}^s]} \cdot X_{\{j_i\}} \ge \frac{1}{2},
  \]
  cf.~\eqref{eq:Aj-lower-bound}, and the inequality $1+x \le e^x$. This completes the proof of Claim~\ref{claim:Ex-Aj-XI}.

  \medskip

  With Claim~\ref{claim:Ex-Aj-XI} now in place, define for every $k \in J$ and every $I \subseteq J$,
  \[
  Z_{k,I} = \indicator[A_{-1}] \cdot \prod_{j \in J, j < k} \indicator[A_j] \cdot X_I
  \]
  and note that if $I \subseteq \{0, \ldots, \ch{k}-1\}$, then $Z_{k,I}$ is $\FF_{s,k}$-measurable. We shall prove the following estimate using induction on $k$.
  \begin{claim}
    \label{claim:Ex-ZkI-bound}
    For every $k \in J$ and every $I \subseteq J \cap \{k, \ldots, \ch{k}-1\}$,
    \begin{equation}
      \label{eq:Ex-ZkI-bound}
      \Ex[ Z_{k,I} \mid \FF_s ] \le \prod_{j \in J, j < k} P_j \cdot \exp\left(\frac{8D}{\alpha^3} \cdot \sum_{j \in (J \cap \{0, \ldots, k-1\}) \cup I} \frac{d_j}{(n-\ch{j})^2}\right).
    \end{equation}
  \end{claim}

  As the statement of Claim~\ref{claim:Ex-ZkI-bound} might look somewhat mysterious at first sight, let us now show how it implies the bound on $\Pr(E_{i,u,w,}^s \mid \FF_s)$ claimed in the statement of Lemma~\ref{lemma:uniform-dist}. To this end, observe that
  \begin{equation}
    \label{eq:sum-j-deg-n-kj-squared}
    \begin{split}
      \sum_{j \in J, j < i} \frac{d_j}{(n-\ch{j})^2} & \le \sum_{j \in J} \frac{d_j}{(n-\ch{j})^2} \le \sum_{j \in J} \sum_{d=0}^{d_j-1} \frac{1}{(n-\ch{j}-d)^2} = \sum_{d=1}^{m-1} \frac{1}{(n-d)^2} \\
      & \le \sum_{d=1}^{m-1} \frac{1}{(n-d)(n-d-1)} = \frac{1}{n-m} - \frac{1}{n-1} \le \frac{1}{\alpha n}.
    \end{split}
  \end{equation}
  Now, as $Z_{i,\emptyset} = \indicator[E_{i,u,w}^s]$, then Claim~\ref{claim:Ex-ZkI-bound}, \eqref{eq:prod-Pj}, and~\eqref{eq:sum-j-deg-n-kj-squared} yield (recalling that $D = 2np$),
  \[
  \begin{split}
    \Pr(E_{i,u,w}^s \mid \FF_s) & \le \prod_{j \in J, j < i} P_j \cdot \exp\left(\frac{8D}{\alpha^3} \cdot \sum_{j \in J, j < i} \frac{d_j}{(n-\ch{j})^2}\right) \le \frac{n-\ch{i}}{(n-1)^2} \cdot \exp\left(\frac{2D}{\alpha n} + \frac{8D}{\alpha^4n} \right) \\
    & \le \frac{n-\ch{i}}{n^2} \cdot \exp\left(\frac{4p}{\alpha} + \frac{16p}{\alpha^4} + \frac{2}{n-1}\right) \le \frac{n-\ch{i}}{n^2} \cdot e^{\delta}.
  \end{split}
  \]
  Therefore, in order to complete the proof, it suffices to prove Claim~\ref{claim:Ex-ZkI-bound}.

  \medskip

  We prove the claim using induction on $k$. For the base case $k=0$, note that $\ch{0} = 1$ and fix some $I \subseteq \{0\}$. Clearly, $Z_{0,I} \le X_I$ and hence~\eqref{eq:Ex-ZkI-bound} follows directly from~\eqref{eq:XI-minus-one-bound} and the inequality $1 + x \le e^x$. Assume now that $k > 0$ and fix an $I$ as above. Let $\ell = \pre{k}$ be the predecessor of $k$ in $J$ in the BFS ordering (so that $\ch{\ell}+d_\ell = \ch{k}$) and let $I' = I \cap \{\ch{\ell}, \ldots, \ch{\ell} + d_\ell-1\}$. Note that
  \[
  Z_{k,I} = Z_{\ell,I \setminus I'} \cdot \indicator[A_\ell] \cdot X_{I'}
  \]
  and that $I \setminus I' \subseteq \{k, \ldots, \ch{\ell}-1\}$. In particular, $Z_{\ell, I \setminus I'}$ is $\FF_{s,\ell}$-measurable and consequently,
  \begin{equation}
    \label{eq:ZkI-induction}
    \Ex[ Z_{k,I} \mid \FF_s ] = \Ex[ \Ex[Z_{k,I} \mid \FF_{s,\ell}] \mid \FF_s] = \Ex[ Z_{\ell,I \setminus I'} \cdot \Ex[ \indicator[A_\ell] \cdot X_{I'} \mid \FF_{s,\ell}] \mid \FF_s ].
  \end{equation}
  Let $A_{<\ell}^* = A_{-1} \cap \bigcap_{j < \ell} A_j$. Observe that if $\ell \le j_i$, then $A_{<\ell}^*$ implies that $u \in U_\ell^s$. Similarly, if $\ell \le i$, then $A_{<\ell}^*$ implies that $w \in U_\ell^s$. In particular if $\ell < i$, then on $A_{<\ell}^*$, we have $\indicator[u \in U_\ell^s] + \indicator[w \in U_\ell^s] = 1 + \indicator[\ell < j_i]$. Recalling~\eqref{eq:Pj-def} and considering separately the three cases: $\ell < j_i$, $\ell = j_i$, and $\ell > j_i$, one can easily see that Claim~\ref{claim:Ex-Aj-XI} implies that
  \[
  \Ex[ \indicator[A_\ell] \cdot X_{I'} \mid \FF_{s,\ell}] \le P_\ell \cdot X_{\{\ell\}} \cdot \exp\left(\frac{8D}{\alpha^3} \cdot \sum_{j \in I'} \frac{d_j}{(n-\ch{j})^2} \right).
  \]
  Substituting the above into~\eqref{eq:ZkI-induction}, we obtain
  \[
  \Ex[ Z_{k,I} \mid \FF_s ] \le P_\ell \cdot \exp\left(\frac{8D}{\alpha^3} \cdot \sum_{j \in I'} \frac{d_j}{(n-\ch{j})^2} \right) \cdot \Ex[Z_{\ell, I \setminus I' \cup \{\ell\}} \mid \FF_s].
  \]
  As $I \setminus I' \cup \{\ell\} \subseteq J \cap \{\ell, \ldots, \ch{\ell}-1\}$, we may use the inductive assumption with $k \leftarrow \ell$ and $I \leftarrow I \setminus I' \cup \{\ell\}$ to obtain
  \[
  \Ex[ Z_{k,I} \mid \FF_s ] \le \prod_{j \in J, j < \ell} P_j \cdot P_\ell \cdot \exp\left(\frac{8D}{\alpha^3} \cdot \sum_{j \in (J \cap \{0, \ldots, \ell-1\}) \cup (I \setminus I') \cup \{\ell\} \cup I'} \frac{d_j}{(n-\ch{j})^2} \right),
  \]
  which is exactly the claimed inequality, as $\ell \in J$ and $\ell = k^-$.
\end{proof}

Finally, we establish the assertion of Remark~\ref{remark:Ui-distribution}. We first argue that a fairly straightforward modification of the proof of Lemma~\ref{lemma:uniform-dist} gives the following estimate.

\begin{lemma}
  \label{lemma:uniform-dist for garbage}
  For every $s\in [N]$ and every pair of distinct $u, w \in V$, the following holds:
  \begin{equation}
    \label{eq:uniform-dist for garbage}
    \Pr\big( u,w \notin \varphi_s(V(T_s)) \wedge \DD_{s-1} \mid \FF_s \big) \le \left(\frac{n-|V(T_s)|}{n-1}\right)^2 \cdot e^\delta.
  \end{equation}
\end{lemma}
\begin{proof}[Proof sketch]
  We argue almost exactly as in the proof of Lemma~\ref{lemma:uniform-dist} with just a few minor modifications. Let $E^s$ denote the event $u,w \not\in \varphi_s(V(T_s))$. We define $A_{-1}$ to be the event $\varphi_s(v_0) \not\in \{u,w\}$ and for every $j \in J$, we let $A_j$ denote the event that $u,w \not\in \{\varphi_s(v_{\ch{j}}), \ldots, \varphi_s(v_{\ch{j}+d_i-1}) \}$, that is, $u$ and $w$ are not among the images of the $d_j$ children of $v_j$ in $T_s$. One immediately sees that
  \[
  E^s = A_{-1} \cap \bigcap_{j \in J} A_j.
  \]
  Item~\ref{item:Ex-Aj-XI} of Claim~\ref{claim:Ex-Aj-XI} is still valid and hence a straightforward modification of Claim~\ref{claim:Ex-ZkI-bound} and of the argument following it gives the estimate
  \[
  \Ex[ E^s \mid \FF_s ] \le \prod_{j \in J} P_j \cdot \exp\left(\frac{8D}{\alpha^3} \cdot \sum_{j \in J} \frac{d_j}{(n-\ch{j})^2}\right),
  \]
  where now $P_j = \left(1 - \frac{d_j}{n-\ch{j}}\right)^2$ for each $j \in J$. We conclude as in the proof of Lemma~\ref{lemma:uniform-dist}, noting additionally that with our new definition of $P_j$, we have the identity
  \[
  \prod_{j \in J} P_j = \left(\frac{n-|V(T_s)|}{n-1}\right)^2.\qedhere
  \]
\end{proof}

We also observe that for every $s \in [N]$, every $v \in V(T_s)$, and every pair of distinct $u, w \in V$,
\begin{equation}
  \label{eq:Pr-phi-vs-is-u-or-w}
  \Pr\big( \varphi_s(v) \in \{u,w\} \wedge \DD_{s-1} \mid \FF_s \big) \le \frac{3}{\alpha n}.
\end{equation}
Indeed, this is clear when $v = v_0$, as $\varphi_s(v_0)$ is a uniformly chosen random element of $V$ and hence the left-hand side of~\eqref{eq:Pr-phi-vs-is-u-or-w} is at most $2/n$. Otherwise, $v$ is the child of some $v_j$ with $j \in J$ and therefore by Claim~\ref{claim:Nisone-size},
\[
\begin{split}
  \Pr\big( \varphi_s(v) \in \{u,w\} \wedge \DD_{s-1} \mid \FF_s \big) & = \Pr\big( \varphi_s(v) \in \{u,w\} \wedge \DD_{s-1} \mid \FF_{s,j} \big) \\
  & = \frac{\indicator[u \in N_j^s(1)] + \indicator[w \in N_j^s(1)]}{|N_j^s(1)|} \le \frac{2}{\alpha n - 2np} \le \frac{3}{\alpha n}.
\end{split}
\]

Let $W_s$ be the set defined in Remark~\ref{remark:Ui-distribution} and observe that $|W_s| = n - |V(T_s)| + 1 \ge \alpha n$. Combining Lemma~\ref{lemma:uniform-dist for garbage} and~\eqref{eq:Pr-phi-vs-is-u-or-w}, we obtain
\begin{equation}
  \label{eq:Pr-u-w-in-Us}
  \Pr\big(\{u,w\} \subseteq W_s \wedge \DD_{s-1} \mid \FF_s\big) \le \left(\frac{|W_s|-1}{n-1}\right)^2 e^\delta + \frac{3}{\alpha n} \le 2 \left(\frac{|W_s|}{n}\right)^2.
\end{equation}

\medskip

Finally, for each $s \in [N]$ and all pairs of distinct $u,w \in V$, we define
\[
Y_{uw,s} = \indicator[\{u,w\} \subseteq W_s \wedge \DD_{s-1}] \cdot \frac{1}{|W_s|}.
\]
As $Y_{uw,s}$ is clearly $\FF_{s+1}$-measurable, inequalities~\eqref{eq:Pr-u-w-in-Us} and $|W_s| \ge \alpha n$ readily imply that the random variables $Y_{uw,1}, \dotsc, Y_{uw,N}$ satisfy the assumptions of Lemma~\ref{lemma:Bennett-plus} with
\[
M \leftarrow \frac{1}{\alpha n}, \qquad \mu \leftarrow \frac{2\max_s|W_s|}{n^2}, \qquad \text{and} \qquad \sigma^2 \leftarrow \frac{2}{n^2}.
\]
Furthermore, letting $t = p \cdot \max_s|W_s| / n$, we see that
\[
\frac{2p}{n} \cdot \max_s|W_s| - N \mu \ge \frac{(1+\eps)p}{n} \cdot \max_s|W_s| \ge t
\]
and also
\[
\frac{Mt}{3} \le \frac{p \cdot \max_s|W_s|}{3\alpha n^2} \qquad \text{and} \qquad N \sigma^2 \le \frac{np}{2} \cdot \frac{2}{n^2} = \frac{p}{n} \le \frac{p \cdot \max_s |W_s|}{\alpha n^2}.
\]
Therefore, it follows from Lemma~\ref{lemma:Bennett-plus} that
\[
  \label{eq:Pr-sum-taunorm}
  \begin{split}
    \Pr\left(Y_{uw,1} + \ldots + Y_{uw,N} > \frac{2p}{n} \cdot \max_s |W_s| \right) & \le \exp\left(- \frac{t^2}{2(N\sigma^2 + Mt/3)} \right) \\
    & \le \exp\left(-\frac{p \alpha \cdot \max_s |W_s|}{3}\right) \le \exp\left(-\frac{p\alpha^2n}{3}\right) \le n^{-10},
  \end{split}
\]
provided that $p \ge 30 \log n/(\alpha^2 n)$. In particular, with probability at least $1 - n^{-8}$, on the event $\DD_N$, every pair of distinct vertices $u,w \in V$ satisfies (recall that $\DD_{s-1} \subseteq \DD_N$ for every $s \in [N]$),
\[
\sum_{s=1}^N \indicator[\{u,w\} \subseteq W_s] \cdot \frac{1}{|W_s|} \le \frac{2p}{n} \cdot \max_s |W_s|.
\]
Finally, as the event $\DD_N$ holds with probability at least $1 - 3n^{-8}$, see~\eqref{eq:Pr-DD-N}, the assertion of Remark~\ref{remark:Ui-distribution} follows.

\section{Concluding remarks}
\label{sec:concluding-remarks}

\begin{itemize}
\item
  The main contribution of this work is the description and the analysis of a randomised algorithm that packs a collection of at most $(1-\eps)np/2$ trees, each of which has at most $(1-\alpha)n$ vertices and maximum degree at most $\Delta$ into the binomial random graph $\Gnp$. It is natural to ask how well our algorithm performs and how tight our analysis is. In the case when both $\alpha$ and $\eps$ are constant, we manage to find a packing under the rather weak assumption that $\Delta < cnp/\log n$ for some positive $c$ that depends only on $\alpha$ and $\eps$. In fact, this is the natural limit of our method (and very likely, also the limit of many other randomised packing strategies), as we shall now argue.

  Suppose that we run our randomised packing algorithm (described in Section~\ref{sec:proof of main}) on a family  of $\lfloor np/4 \rfloor$ trees, each of which has between $n/2$ and $3n/4$ vertices, whose all degrees are either $1$ or $\Delta := \lceil np/\omega \rceil$ for some $\omega = \omega(n) > 4$; clearly, such trees exist. We argue that our algorithm will fail to pack these trees into $\Gnp$ unless $\omega(n) \ge c\log n$ for some positive constant $c$. To see this, note that each tree in the collection contains at least $\omega / (3p)$ vertices of degree $\Delta$. In particular, in any packing of the trees into an $n$-vertex graph, an $\omega/(3pn)$-proportion of the vertices of the host graph will have degree $\Delta$ in the image of any given tree. Now, observe that our randomised packing algorithm has the following nice property. In each of the rounds, a given vertex of the currently embedded tree is mapped to a given vertex of the host graph with probability at most around $1/n$, independently of the earlier rounds; this is an easy consequence of Lemma~\ref{lemma:uniform-dist}. It follows that in each of the rounds, most vertices of the host graph are the images of a vertex of degree $\Delta$ with probability at least $\omega/(4np)$. Therefore, for a typical vertex $v$ in the host graph, the probability that $v$ serves as a vertex of degree $\Delta$ more than $2\omega$ times is at least $e^{-C\omega}$ for some absolute constant $C$. Hence, if $\omega \ll \log n$, then some vertices in the host graph will accumulate total degree of at least $2\omega \cdot \Delta \ge 2np$, which clearly does not usually happen in $\Gnp$.

\item
  While writing the proof of Theorem~\ref{theorem:main}, we were much less concerned with the optimality of the assumptions listed in~\eqref{eq:main-assumptions} with respect to $\alpha$ and $\eps$, settling for a polynomial dependence on both these parameters, wich then results in an upper bound of the form $(np)^{c} / (\log n)^C$ on the maximum degree of the trees in Theorem~\ref{theorem:spanning}. The current value $c = 1/6$ could be improved to any constant smaller than $1/5$ if one replaced the $3/2$ in the estimate~\eqref{eq:Pr-Aj-FFsj-error-term} by a smaller constant larger than $1$. We decided not to do this for the sake of clarity of the presentation. One could most likely improve the estimate~\eqref{eq:sum-degvk-nkk} by using the inequality $|I| \le d_j \le \Delta$. Again, we decided not to pursue this direction, as this could only really affect the case $\Delta \ll \sqrt{n}$. It would be extremely interesting to relax the assumption $\Delta \ll (np)^{1/2}$ of Theorem~\ref{theorem:spanning-weaker}, even for small values of $p$, as this would most likely require far-reaching improvements of our packing strategy.

\item
  It is plausible that one could improve our algorithm to produce a packing of trees with maximum degree as large as $\Theta(np)$ in $\Gnp$. For example, one can try, in each time step, to map vertices of ``high'' degrees in the tree to vertices of ``small'' degrees in the current embedding. This would prevent vertices from begin images of high degree vertices too often, and could potentially remove the $1/\log n$ factor from the current upper bound on $\Delta$. Having said that, the analysis of such an algorithm would most likely differ significantly from our current analysis (and would probably be much more complicated). Since anyway we do not believe that such a naive random procedure will resolve Conjecture~\ref{tpc}, we did not try to continue this argument. Still, it would be very interesting to see a clean analysis of an algorithm of a similar type.

\item
  Our embedding scheme relies very strongly on the fact that we embed only graphs that are $1$-degenerate (recall that a graph $H$ is $d$-degenerate if and only if there exists a labeling $v_1, \dotsc, v_m$ of $V(H)$ for which every $v_i$ has at most $d$ neighbours among $v_1,\dotsc, v_{i-1}$). Indeed, following such an ordering, in each time step we try to embed a new vertex by exposing exactly one new edge, and therefore the algorithm is not ``wasteful'' and leaves us a lot of ``randomness'' for later steps. It would be very interesting to find random embedding schemes employing the ``online sprinkling'' idea for general graphs or, at the very least, for almost-spanning graphs with bounded maximum degree.
\end{itemize}

\medskip
\noindent
\textbf{Acknowledgment.}
We would like to thank Choongbum Lee for many stimulating discussions and ideas on an earlier version of this paper.

\bibliographystyle{amsplain}
\nocite{*}
\bibliography{PackingTreesInGnp}


\appendix

\section{Proof of Lemma~\ref{lemma:Bennett-plus}}

\label{sec:proof-Bennett-plus}

In the proof of Lemma~\ref{lemma:Bennett-plus}, we shall use the following (conditional version of the) estimate on the moment generating function of a bounded random variable with bounded second moment.

\begin{lemma}[{\cite{Be62}; see also~\cite[Theorem~2.9]{BoLuMa13}}]
  \label{lemma:Bennett-MGF}
  Let $X$ be a random variable satisfying
  \[
  0 \le X \le M, \qquad \Ex[X \mid \FF] \le \mu, \qquad \text{and} \qquad \Ex[X^2 \mid \FF] \le \sigma^2
  \]
  for some $\sigma$-field $\FF$ and reals $M$, $\mu$, and $\sigma$. Then, for all $\lambda \in \RR$,
  \[
  \Ex\left[e^{\lambda X} \mid \FF\right] \le \exp\left(\mu \lambda + \frac{\sigma^2}{M^2} \phi(M\lambda) \right).
  \]
  where $\phi(x) = e^x - x - 1$.
\end{lemma}

\begin{proof}[Proof of Lemma~\ref{lemma:Bennett-plus}]
  Suppose that $X_1, \ldots, X_N$ and $M$, $\mu$, and $\sigma$ satisfy the assumptions of the lemma. We prove the claimed upper tail estimate using a standard Azuma-type argument. We first derive an upper bound on the moment generating function of $\sum_{i=1}^N X_i$. As in the statement of Lemma~\ref{lemma:Bennett-MGF}, let $\phi(x) = e^x - x - 1$.
  \begin{claim}
    \label{claim:Azuma-Bennett}
    For all $\lambda \ge 0$ and $i \in \{0, \ldots, N\}$, we have
    \begin{equation}
      \label{eq:Azuma-Bennett}
      \Ex\left[\exp\left(\lambda \sum_{j=1}^i X_j\right)\right] \le \exp\left(\mu\lambda + \frac{\sigma^2}{M^2}\phi(M\lambda) \right)^i.
    \end{equation}
  \end{claim}
  We prove the claim by induction on $i$. For $i = 0$, there is nothing to prove, so assume that $i \ge 1$ and that~\eqref{eq:Azuma-Bennett} holds with $i$ replaced by $i-1$. Note that
  \begin{equation}
    \label{eq:conditional-Ex}
    \Ex\left[\exp\left(\lambda \sum_{j=1}^i X_j\right)\right] = \Ex\left[\exp\left(\lambda \sum_{j=1}^{i-1} X_j\right) \cdot \Ex\left[e^{\lambda X_i} \mid X_1, \ldots, X_{i-1} \right]\right].
  \end{equation}
  Using Lemma~\ref{lemma:Bennett-MGF} to bound the conditional expectation in the right-hand side of~\eqref{eq:conditional-Ex}, we obtain
  \[
  \Ex\left[\exp\left(\lambda \sum_{j=1}^i X_j\right)\right] \le \Ex\left[\exp\left(\lambda \sum_{j=1}^{i-1} X_j\right) \cdot  \exp\left(\mu \lambda + \frac{\sigma^2}{M^2} \phi(M\lambda) \right)\right],
  \]
  which together with our inductive assumption immediately gives~\eqref{eq:Azuma-Bennett}.

  \medskip

  With the upper bound~\eqref{eq:conditional-Ex} in place, we use the Cram\'er--Chernoff method to obtain the claimed estimate for the upper tail. Indeed, for all positive $\lambda$ and $t$, by Markov's inequality,
  \begin{align*}
    \Pr\left( \sum_{j=1}^N X_j \ge N\mu + t \right) & =   \Pr\left( \exp\left( \lambda \sum_{j=1}^N X_j\right) \ge e^{\lambda(N\mu + t)} \right) \\
                                                    & \le e^{-\lambda(N\mu + t)} \Ex\left[\exp\left( \lambda \sum_{j=1}^N X_j\right)\right] \le \exp\left(\frac{N\sigma^2}{M^2} \phi(M\lambda) - \lambda t\right),
  \end{align*}
  where the last inequality is~\eqref{eq:Azuma-Bennett} with $i = N$. Letting $u = Mt/(N\sigma^2)$ and $\lambda = \log(1+u) / M$, we obtain
  \begin{equation}
    \label{eq:upper-bound-raw}
    \Pr\left( \sum_{j=1}^N X_j \ge N\mu + t \right) \le \exp\left(-\frac{N\sigma^2}{M^2} \cdot \big((1+u)\log(1+u)-u\big)\right).
  \end{equation}
  Finally, the claimed estimate follows from the following easy-to-prove inequality:
  \[
  (1+u) \log(1+u) - u \ge \frac{u^2}{2(1+u/3)} \qquad \text{for all $u > 0$}.\qedhere
  \]
\end{proof}

\end{document}